\documentclass[reqno]{amsart}
\renewcommand{\marginpar}[1]{\relax}

\usepackage{tikz}
\usepackage{amsmath,amscd}
\usepackage{amssymb,amsfonts,mathrsfs}
\usepackage[backref,colorlinks=true,linkcolor=blue,citecolor=blue]{hyperref}
\usepackage{math60,mlthm}
\usepackage{local}

\begin{document}
\title[Local global principle for regular operators]{A local global principle
for regular operators in Hilbert $C^*$--modules}

\author{Jens Kaad}
\address{Hausdorff Center for Mathematics,
Universit\"at Bonn,
Endenicher Allee 60,
53115 Bonn,
Germany}
\email{jenskaad@hotmail.com}

\author{Matthias Lesch}
\address{Mathematisches Institut,
Universit\"at Bonn,
Endenicher Allee 60,
53115 Bonn,
Germany}

\email{ml@matthiaslesch.de, lesch@math.uni-bonn.de}
\urladdr{www.matthiaslesch.de, www.math.uni-bonn.de/people/lesch}

\date{March 8, 2012}

\thanks{Both authors were supported by the 
        Hausdorff Center for Mathematics, Bonn}



\begin{abstract}
Hilbert $C^*$--modules are the analogues of Hilbert spaces where a
$C^*$--algebra plays the role of the scalar field. 
With the advent of Kasparov's celebrated $KK$--theory they 
became a standard tool in the theory of operator algebras.

While the elementary properties of Hilbert $C^*$--modules can be derived
basically in parallel to Hilbert space theory the lack of an analogue of the
Projection Theorem soon leads to serious obstructions and difficulties.

In particular the theory of unbounded operators is notoriously more complicated
due to the additional axiom of regularity which is not easy to check.

In this paper we present a new criterion for regularity in terms of the 
Hilbert space localizations of an unbounded operator. We discuss several 
examples which show that the criterion can easily be checked and that it 
leads to nontrivial regularity results.
\end{abstract}

\maketitle
\tableofcontents
\listoffigures

\section{Introduction} 
A Hilbert $C^*$--module $E$ over a $C^*$--algebra $\sA$ is an $\sA$--right
module equipped with an $\sA$--valued inner product $\inn{\cdot,\cdot}$ and
such that $E$ is complete with respect to the norm
$\|x\|:=\|\inn{x,x}^{1/2}\|=\|\inn{x,x}\|^{1/2}$. The notion was introduced by
Kaplansky in the commutative case \cite{Kap:MOA} and in general independently
by Paschke \cite{Pas:IPM}, Rieffel \cite{Rie:IRC} and Takahashi (for the latter
cf. \cite[p. 179]{Rie:IRC} and \cite[p. 364]{Hof:RAC}). Kasparov's celebrated
$KK$--theory makes extensive use of Hilbert $C^*$--modules \cite{Kas:HCM} and
by now Hilbert $C^*$--modules are a standard tool in the theory of operator
algebras. They
are covered in several textbooks, Blackadar \cite[Sec. 13]{Bla:KTO2Ed},
\cite[Sec. II.7]{Bla:OA}, Manuilov-Troitsky \cite{ManTro:HCM}, Raeburn-Williams
\cite{RaeWil:MEC}, Wegge-Olsen \cite[Chap. 15]{Weg:KTC}; our standard reference
will be Lance \cite{Lan:HCM}.

The elementary properties of $C^*$--modules can be derived basically in
parallel to Hilbert space theory. However, there is no analogue of the
Projection Theorem which soon leads to serious obstructions and difficulties.

A Hilbert $C^*$--module $E$ comes with a natural $C^*$--algebra $\B(E)$
of bounded adjointable module endomorphisms. As for Hilbert spaces
one soon needs to consider unbounded adjointable operators,
Baaj-Julg \cite{BaaJul:TBK}, 
Gulja{\v{s}} \cite{Gul:UOH},
Kucerovsky \cite{Kuc:KKP},  Pal \cite{Pal:ROH}, Woronowicz \cite{Wor:UEA};
see also \cite[Chap. 9/10]{Lan:HCM}.

The lack of a Projection Theorem in Hilbert $C^*$--modules causes
the theory of unbounded operators to be notoriously more complicated.
To explain this let us introduce some terminology: following \textsc{Pal} 
\cite{Pal:ROH} by a \emph{semiregular}
operator in a Hilbert $C^*$--module $E$ over $\sA$
we will understand an operator $T:\dom(T)\longrightarrow E$ 
defined on a dense $\sA$--submodule $\dom(T)\subset E$
and such that the adjoint $T^*$ is densely defined, too.
One now easily deduces that $T$ is $\sA$--linear and closable and that
$T^*$ is closed. Besides this semiregular operators can be rather pathologic
(see the discussion in Sec. \plref{ss:SRO} and Sec. \plref{s:NRO}).

To have a reasonable theory (e.g. with a functional calculus for selfadjoint
operators) one has to introduce the additional \emph{axiom of regularity}:
a closed semiregular operator $T$ in $E$ is called \emph{regular} if
$I+T^*T$ is invertible. Regular operators behave more or less as nicely
as closed densely defined operators in Hilbert space. In particular
for selfadjoint regular operators there is a continuous
functional calculus \cite{Baa:MNB}, \cite{Kuc:FCR}, \cite{Wor:UEA}, \cite{WorNap:OTC}. 

While in a \emph{Hilbert space} every densely defined closed operator is regular
in general Hilbert $C^*$--modules there exist closed semiregular operators which
are not regular, see Prop. \plref{p:RegT}.

There is, however, a considerable drawback of the regularity axiom. 
We quote here from \cite[p. 332]{Pal:ROH}:

\begin{quote}
... But when one deals with specific unbounded operators on concrete Hilbert
$C^*$--modules, it is usually extremely difficult to verify the regularity
condition, though the semiregularity conditions are relatively easy to
check. So it would be interesting to find other more easily manageable
conditions that are equivalent to the last condition above. In \cite{Wor:UEA},
Woronowicz gave a criterion based on the graph of an operator for it to be
regular, and to this date, this remains the only attempt in this direction.
\end{quote}

The aim of this paper is to remedy this distressing situation which has
not much improved in the more than 10 years after Pal had written this.
Before going into that let us briefly comment on Woronowicz's work and
explain the criterion mentioned in the previous paragraph.

Woronowicz \cite{Wor:UEA} works in the a priori special situation $E=\sA$,
i.e. $E$ is the $C^*$--algebra $\sA$ viewed as a Hilbert module over
itself. This is not as special as it seems: for a general Hilbert
$C^*$--module $E$ it is shown in \cite[Sec. 3]{Pal:ROH} that there is a one-one
correspondence between (semi)regular operators in $E$ and (semi)regular
operators in the $C^*$--algebra $\sK(E)$ of $\sA$--compact operators.
Nevertheless, we do not quite agree with loc. cit. that this fact
allows ``without any loss in generality" to restrict one--self to
(semi)regular operators on $C^*$--algebras. After all changing the
scalars from $\sA$ to $\sK(E)$ is rather substantial. 

Woronowicz' criterion reads as follows: a closed semiregular
operator $T$ in $E$ is regular if and only if its graph
\[
     \Gamma(T):=\bigsetdef{(x,y)\in E\oplus E}{ x\in\dom(T), y=Tx}.
\]
is complementable. This was proved in \cite{Wor:UEA} for $E=\sA$,
the (straightforward) extension to the general case can be found
in \cite[Theorem 9.3 and Prop. 9.5]{Lan:HCM}.

In practical terms this criterion does not help much. One rather quickly sees
that checking it boils down to solving the equation
$(I+T^*T)x=y$.

Let us describe in non--technical terms the problem from which
this paper arose. In our study of an approach to the $KK$--product
for unbounded modules \cite{KaaLes:SFU} we needed to study two selfadjoint regular
operators $S,T$ in a Hilbert $C^*$--module with ``small" commutator
(Section \plref{s:SSRO}). More precisely, we were looking at unbounded odd
Kasparov modules $(D_1,X)$ and $(D_2,Y)$ together with a densely defined
connection $\nabla$. The operator $S$ then corresponds to $D_1 \ot 1$ whereas $T$
corresponds to $1 \ot_{\nabla} D_2$. The Hilbert $C^*$--module is given by the
interior tensor product of $X$ and $Y$ over some $C^*$--algebra. As an
essential part of forming the unbounded Kasparov product of $(D_1,X)$ and
$(D_2,Y)$ one needs to study the selfadjointness and regularity of the
unbounded product operator
\begin{equation}\label{intro-eq:DefD}
D := \begin{pmatrix}
0 & S - i\, T  \\
S + i\, T & 0 
\end{pmatrix}, \quad  \dom (D) = \big( \dom (S) \cap \dom (T) \big)^2\subset
E\oplus E.
\end{equation}
With some effort we could prove that this operator is selfadjoint but all
efforts to prove regularity using Woronowicz' criterion failed. For a while we
even started to look for counterexamples. On the other hand, in a Hilbert
space regularity comes for free and the construction of $D$ out of $S$ and $T$
was more or less ``functorial".

So stated somewhat vaguely, the following principle should
hold true: given a ``functorial" construction of an operator
$D=D(S,T)$ out of two selfadjoint and regular operators $S,T$.
If then for Hilbert spaces this construction always produces a
selfadjoint operator then $D(S,T)$ is selfadjoint and regular.

More rigorously, let us consider a closed, densely defined and, for simplicity,
symmetric operator $T$ in the Hilbert $C^*$--module $E$. Then for
each state $\go$ on $\sA$ there is a canonical Hilbert space $E^\go$, a natural
map $\iom:E\to E^\go$ with dense range, and a symmetric operator 
$T^\go$ which is defined by closing the operator defined
by $T^\go_0 \iom(x):=\iom(Tx)$. We call $T^\go$ the \emph{localization}
of $T$ with respect to the state $\go$. One of the main results
of this paper is the following Local--Global Principle. For
the sake of brevity, it is stated here for symmetric operators. See
Theorem \plref{t:locglob} for the general case.

\begin{theorem}[Local--Global Principle]\label{intro-t:locglob}
For a closed, densely defined and symmetric operator $T$
the following statements are equivalent:
\begin{thmenum}
\item $T$ is selfadjoint and regular.
\item For every state $\go\in S(\sA)$ the localization $T^\go$ is selfadjoint.
\end{thmenum}
\end{theorem}

The main tool for proving this Theorem is the following separation
Theorem.

\begin{theorem}\label{intro-t:sep} 
Let $L \subset E$ be a closed convex subset of the
Hilbert $C^*$--module $E$ over $\sA$. For each vector $x_0\in E\setminus L$ there
exists a state $\go$ on $\sA$ such that $\iom(x_0)$ is not in the closure of
$\iom(L)$. In particular there exists a state $\go$ such that $\iom(L)$ is
not dense in $E^\go$ and hence $\iom(L)^\perp\not=\{0\}$. 
\end{theorem}

We will show by a couple of examples that the Local--Global Principle
can easily be checked in concrete situations. We would find it
aesthetically more appealing if in Theorems \plref{intro-t:locglob}
and \plref{intro-t:sep} one could replace ``state" by ``pure state".
We conjecture that this is true, but we can only prove it under additional
assumptions on the Hilbert $C^*$--module. That pure states suffice in these
cases turns out to be practically useful in Sections \plref{s:Pure} and in the
discussion of examples of nonregular operators in Section \plref{s:NRO}. We
therefore single out the following Conjecture:
\begin{conjecture}\label{intro-conj}
If $L$ is a proper submodule of the $C^*$--algebra $\sA$ then
there exists a pure state $\go$ on $\sA$ such that
$\iom(L)^\perp\not=\{0\}$.

Consequently, a closed densely defined symmetric operator
in the Hilbert $C^*$--module $E$ over $\sA$ is regular if and
only if for each pure state $\go$ on $\sA$ the localization
$T^\go$ is selfadjoint.
\end{conjecture}

We close this introduction with a few
remarks about the organization of the paper:

In Section \plref{s:ROL} we collect the necessary background and
notation. In particular (semi)regular operators and their
localizations with respect to representations of the underlying
$C^*$--algebra are introduced. 

In Section \plref{sepcstar} we prove the separation Theorem
\plref{intro-t:sep} and discuss various corner cases which illustrate that
the separation Theorem is not as obvious as it might seem.
As a first application we show that a
submodule $\sE\subset\dom(T)$ is a core for $T$ if and only if for each state
$\go$ the subspace $\iom(\sE)\subset \dom(T^\go)$ is a core for the localized
operator $T^\go$ (Theorem \plref{ss:CCS}).

Section \plref{s:LGP} contains the statement and proof of the
main result of this paper, the Local--Global Principle
characterizing the regularity of a semiregular operator $T$ in terms
of the Hilbert space localizations $T^\go$. The proof uses crucially
the separation Theorem \plref{t:sep} in Section \plref{sepcstar}.
To illustrate the power of the Local--Global Principle we generalize W\"ust's extension
of the Kato--Rellich Theorem to Hilbert $C^*$--modules (Theorem
\plref{t:Wust}).

Section \plref{s:Pure} discusses various aspects of the conjectural refinement
of the Local--Global Principle, Conjecture \plref{intro-conj}.
We prove the conjecture for Hilbert $C^*$--modules over commutative
$C^*$--algebras (Theorem \plref{t:Pure}) as well as for the Hilbert
$C^*$--module $E=\sA$ for any $C^*$--algebra (Theorem
\plref{t:InvHMPureState}). Furthermore, it is shown that for a finitely
generated Hilbert module over a commutative $C^*$--algebra every semiregular
operator is regular (Theorem  \plref{t:ROvectorbundles}); this was earlier
proved by Pal \cite[Sec. 4]{Pal:ROH} for the special module $E=\sA$ for $\sA$
commutative.

In Section \plref{s:NRO} we will recast in a slightly more general context the
known constructions for nonregular operators. Propositions \plref{p:RegT} and
\plref{p:NonRegularFaithfulState} give a precise measure theoretic
characterization for the regularity of a large class of semiregular
operators acting on the Hilbert module $C(X,H)$ ($X$ some compact space
and $H$ some Hilbert space). These results contain the known examples
of nonregular operators as special cases. 

Finally, Section \plref{s:SSRO} contains the regularity result
which was the main motivation to write this paper, as explained above. 
We will study the regularity of sums $S\pm iT$ where $S,T$ are selfadjoint
regular operators in some Hilbert $C^*$--module with a technical
condition on the size of the commutator $[S,T]$. 
We will make crucial use of this result in a subsequent publication on the
unbounded Kasparov product, see \cite{KaaLes:SFU}.

\section*{Acknowledgments}
We would like to thank Ryszard Nest for helpful discussions, in particular for 
sharing with us his insight that the localized Hilbert space with respect to an
arbitrary representation is given as the interior tensor product (cf. Sections
\plref{ss:LHC}, \plref{ss:LSO}). We would also like to thank the referee for some valuable remarks on our exposition. Indeed, the alternative and simpler proof of Proposition \ref{p:RegT} was suggested to us by the referee.

\section{Regular operators and their localizations}\label{s:ROL}  
\subsection{Notations and conventions} \label{ss:NC} 
Script letters $\sA,\sB,\ldots$ denote  
(not necessarily unital) $C^*$--algebras. 
Hilbert $C^*$--modules over a $C^*$--algebra will be denoted by
letters $E,F,\ldots$; $H$ usually denotes a Hilbert space, i.e. a Hilbert $C^*$--module
over $\C$. Recall that a Hilbert $C^*$--module over $\sA$ is an $\sA$--right module 
equipped with an $\sA$--valued inner product $\inn{\cdot,\cdot}$. 
Furthermore, it is assumed that $E$ is complete with respect to the
induced norm $\|x\|:=\|\inn{x,x}^{1/2}\|=\|\inn{x,x}\|^{1/2}$. 

We will adopt the convention that inner products are conjugate \emph{$\sA$--linear}
in the \emph{first} variable and \emph{linear} in the \emph{second}. This convention 
is also adopted for Hilbert spaces. 
We let $\B(E)$ denote the $C^*$--algebra of bounded adjointable operators on $E$.
Our standard reference for Hilbert $C^*$--modules is \textsc{Lance} \cite{Lan:HCM}.
\subsection{Localizations of Hilbert $C^*$--modules, cf. \cite[Chap. 5]{Lan:HCM}}\label{ss:LHC}  
Let $\pi$ be a representation of $\sA$ on the Hilbert space $H_\pi$. We then get an
induced representation $\pi_E$ of $\B(E)$ on the \emph{interior} tensor product
$E\hot_\sA H_\pi$ \cite[Chap. 4]{Lan:HCM}. 
The latter is the Hilbert space obtained as the completion
of the algebraic tensor product $E\ot_\sA H_\pi$ with respect to the inner
product
\begin{equation}\label{eq:108}  
     \inn{x\otimes h, x'\otimes h'} = \inn{h, \pi(\inn{x,x'}_\sA) h'},
\end{equation}
and for $T\in\B(E)$ one has $\pi_E(T)(x\otimes h)=(Tx)\ot h$.
We emphasize that by \cite[Prop. 4.5]{Lan:HCM} the inner product
\eqref{eq:108} on $E\ot_\sA H_\pi$ is indeed positive definite and hence 
$E\ot_\sA H_\pi$ may be viewed as a dense subspace of $E\hot_\sA H_\pi$. 
We call the Hilbert space $E\hot_\sA H_\pi$ the \emph{localization} of $E$ with
respect to the representation $\pi$. If $\pi$ is faithful then so is the
induced representation $\pi_E$ of $\B(E)$. 

For cyclic representations one has a slightly different but equivalent
description of $E\hot_\sA H_\pi$. Namely, let $\go\in S(\sA)$ be a state. Then
one can mimic the GNS construction for $E$ as follows:
$\go$ gives rise to a (possibly degenerate) scalar product
\begin{equation}\label{eq:100}  
        \scalar{x}{y}_\go:= \go(\scalar{x}{y})    
\end{equation}
on $E$. $\sN_{\go}:=\bigsetdef{x\in E}{\scalar{x}{x}_\go=0}$ is a subspace
of $E$. $\scalar{\cdot}{\cdot}_\go$ induces a scalar product on
the quotient $E/\sN_{\go}$ and we denote by $E^\go$ the Hilbert
space completion of $E/\sN_{\go}$.
We let $\iom : E \to E^\go$ denote the natural map. Clearly $\iom$ is
continuous with dense range; it is injective if and only if
$\go$ is faithful.

Now let $(\pi_\go,H_\go,\xi_\go)$ be the cyclic representation of $\sA$ 
with cyclic vector $\xi_\go$ associated with the state $\go$. One then
has $\inn{\xi_\om, \pi_\om(a)\xi_\om} = \om(a)$ for $a \in \sA$.
Furthermore, the map
\begin{equation}\label{eq:109}  
       E^\om \to E \hot_A H_\om, \quad \iom(e) \mapsto e \ot \xi_\om
\end{equation}
is a unitary isomorphism. We will from now on tacitly identify $E^\go$
with $E\hot_\sA H_\go$ and hence identify $\iom(e)$ with $e\ot \xi_\go$ 
where convenient.
\subsection{Semiregular and regular operators}\label{ss:SRO}  

Following \textsc{Pal} \cite{Pal:ROH} by a \emph{semiregular}
operator in $E$ we will understand an operator $T:\dom(T)\longrightarrow E$ 
defined on a dense $\sA$--submodule $\dom(T)\subset E$
and such that the adjoint $T^*$ is densely defined, too.

This definition is the adaption of the notion of a densely
defined closable operator in the Hilbert space setting. 
Pal also requires that $T$ is closable but, as for Hilbert spaces, 
this indeed follows from the other assumptions:

\begin{lemma}\label{l:BasicSemiRegular}
Let $T$ be a semiregular operator in $E$. Then $T$ is $\sA$--linear
and closable. The adjoint $T^*$ is closed and $T^*=(\ov T)^*$. Here $\ov T$ denotes the
closure of $T$.
\end{lemma}
\begin{proof} $\sA$--linearity and closability are simple consequences of
the fact that $T^*$ is densely defined. E.g. let $(x_n)\subset \dom(T)$
be a sequence such that $x_n\to 0$ and $Tx_n\to y$. Then for all
$z\in\dom(T^*)$
\[
   \inn{y,z}=\lim_{n\to\infty} \inn{Tx_n,z}=\lim_{n\to\infty}\inn{x_n,T^*z}=0
\]
and hence $y=0$. This proves that $T$ is closable. The remaining claims
follow easily.
\end{proof}

Besides this one should not take for granted any of the
properties one is used to from unbounded operators in Hilbert space.
Semiregular operators in Hilbert $C^*$--modules can be rather pathologic, see e.g. \cite[Chap. 9]{Lan:HCM} and 
Section \plref{s:NRO} below. We mention as a warning that in general
$\ov{T}\subsetneqq T^{**}$, see Prop. \plref{p:RegT} and the discussion thereafter.


A \emph{closed semiregular} operator $T$ is called \emph{regular} if in addition
$I+T^*T$ has dense range. It then follows that $I+T^*T$ is densely defined
\cite[Lemma 9.1]{Lan:HCM} and invertible. Regular operators behave more or less as nicely
as closed densely defined operators in Hilbert space. In particular
for selfadjoint regular operators there is a continuous
functional calculus \cite{Baa:MNB}, \cite{Kuc:FCR}, \cite{Wor:UEA}, \cite{WorNap:OTC}.

Since the functional calculus will be needed, let us briefly describe it.
Let $C_\infty(\R)$ denote the algebra of continuous functions $f$ on the real
line such that $f$ has limits as $x\to \pm \infty$. This algebra is isomorphic
to the continuous functions on the compact interval $[-1,1]$ via
$C_\infty(\R)\ni f\mapsto \tilde f\in C[-1,1], \tilde f(x):=f(x/\sqrt{1-x^2})$. For a selfadjoint regular operator $T$ the bounded transform $T (I+T^2)^{-1/2}$ is in $\B(E)$
(cf. \cite[Chap. 10]{Lan:HCM}). Putting $f(T):=\tilde f(T(I+T^2)^{-1/2})$ then
yields a $*$--homomorphism $C_\infty(\R)\to \B(E)$ which sends the function
$x\mapsto (1+x^2)\ii$ to $(I+T^2)\ii$ and $x\mapsto x(1+x^2)\ii$ to
$T(I+T^2)\ii$.

While in a \emph{Hilbert space} every densely defined closed operator is regular
in general Hilbert $C^*$--modules there exist closed semiregular operators which
are not regular, see Prop. \plref{p:RegT}.

\begin{lemma}[{cf. \cite[Cor. 9.6]{Lan:HCM}}]\label{l:LanceCorrected}
 Let $T$ be a regular operator. Then $T^*$ is regular, too.
Furthermore $T=T^{**}$.
\end{lemma}
\begin{proof} Lance states this as an if and only if condition. 
As pointed out by Pal \cite[Rem. 2.4 (ii)]{Pal:ROH}
Cor. 9.6 in \cite{Lan:HCM} is not correct as stated. 
Indeed \cite[Prop. 2.2 and 2.3]{Pal:ROH} shows that there 
exists a semiregular nonregular symmetric operator $S$ such
that $S^*$ is selfadjoint and regular. 

\marginpar{TODO: comment 2.(10)}
An inspection of the arguments preceding \cite[Cor. 9.6]{Lan:HCM}
shows that the regularity of $T$ indeed implies the regularity
of $T^*$. Furthermore, then $T=T^{**}$ by \cite[Cor. 9.4]{Lan:HCM}. 

In case of the operator $S$ one can still conclude the
regularity of $S^{**}$. There is no contradiction here, it just
follows that $S^{**}\neq \ov S$.
\end{proof}

Symmetry and selfadjointness are defined as usual as $T\subset T^*$ resp. $T=T^*$. The following reduction of
the regularity problem to selfadjoint operators will be convenient. 

\begin{lemma}\label{l:THat} Let $T$ be a closed and semiregular operator and define 
\begin{equation}\label{eq:THat}  
\hat T :=\mat{0}{T^*}{T}{0}.
\end{equation}
Then $\hat T$ is a closed symmetric operator. Moreover, $T$ is regular if and
only if $\hat T$ is selfadjoint \emph{and} regular.
\end{lemma}
\begin{proof} That $\hat T$ is closed and symmetric is immediate.

If $T$ is regular then by Lemma \plref{l:LanceCorrected} $T^*$ is also
regular and $T^{**}=T$. Thus $\hat T$ is selfadjoint and
\begin{equation}\label{eq:111}  
      I+\hat T^2 =I+\hat T^* \hat T=\mat{I+T^* T}{0}{0}{I+ T T^*} =  \mat{I+T^* T}{0}{0}{I+T^{**} T^*} 
\end{equation}
is invertible.

Conversely, if $\hat T$ is selfadjoint and regular then the first two equalities in \eqref{eq:111}
hold and they show that $T$ is regular.
\end{proof}

For closed operators in Hilbert space
the domain equipped with the graph scalar product is in itself a Hilbert space.
We briefly discuss the analogous construction for a semiregular operator $T$.
For $x,y\in\dom(T)$ put
\begin{equation}\label{eq:116}  
    \inn{x,y}_T:= \inn{x,y}+\inn{Tx,Ty}.
\end{equation}
It is straightforward to check that this turns $\dom(T)$ into a pre--Hilbert $C^*$--module
which is complete if and only if $T$ is a closed operator. Furthermore, the natural
inclusion $\iota_T:\dom(T)\hookrightarrow E$ is a continuous $\sA$--module homomorphism.
Furthermore, we have

\begin{prop}\label{p:GraphHilbertModule}
For a closed semiregular operator $T$ the map 
$\iota_T$ is adjointable if and only if $T$ is regular. In that case
one has $\iota_T^*=(I+T^*T)\ii$, where the latter is viewed as a map
$E\longrightarrow \dom(T)$.
\end{prop}

We leave the simple proof to the reader, cf. also \cite[Chap. 9]{Lan:HCM}.
\subsection{Localizations of semiregular operators} \label{ss:LSO}  

Let $E$ be a Hilbert $C^*$--module over some $C^*$--algebra $\sA$. 
Furthermore, let $\pi$ be a representation of $\sA$ on the Hilbert space $H_\pi$.
The construction of $\pi_E$ in Section \plref{ss:LHC} can be extended
to semiregular operators. Let $T$ be a semiregular operator in $E$.
We define $T_0^\pi$ as unbounded operator in $E\hot_\sA H_\pi$
by\marginpar{TODO 2.(11) or skip} 
\begin{equation}\label{eq:Tpi0}
\dom(T^\pi_0) := \dom(T) \ot_\sA H_\pi,\quad T^\pi_0(x\ot h):= (Tx)\otimes h \in E\hot_\sA H_\pi.
\end{equation}
$T_0^\pi$ is certainly well--defined on the dense $\sA$--submodule $\dom(T)\ot_\sA H_\pi$.
Furthermore, for $x\in\dom(T), y\in\dom(T^*), h_1, h_2\in H_\pi$
\begin{equation}\label{eq:112}  
 \begin{split}
  \binn{ T_0^\pi (x\ot h_1), y\ot h_2 } &
      = \inn{ (Tx)\otimes h_1, y\ot h_2} = \inn{h_1, \pi(\inn{T x,y}) h_2} \\
     &= \inn{h_1, \pi(\inn{x, T^* y}) h_2}= \inn{x\ot h_1, (T^*)_0^\pi(y\ot h_2)}.
 \end{split}
\end{equation}
This shows that the densely defined operator $(T^*)_0^\pi$ is contained in
$(T_0^\pi)^*$. Let us summarize

\begin{lemma}\label{l:LocalizedOperator}
For any representation $(\pi,H_\pi)$ of $\sA$ the operator
$T_0^\pi$ is densely defined and closable. Furthermore, $(T^*)_0^\pi\subset (T_0^\pi)^*$.
We let $T^\pi$ be the closure of $T_0^\pi$ and call it the localization of $T$ with
respect to the representation $(\pi,H_\pi)$. We have $(T^*)^\pi\subset (T^\pi)^*$.

In particular if $T$ is symmetric then the localization $T^\pi$ is symmetric, too.
\end{lemma}

Finally we note that if $(\pi_\go, H_\go, \xi_\go)$ is the cyclic representation
associated to the state $\go$ we write $T_0^\go$ resp. $T^\go$ for the localization
viewed as an operator in $E^\go$. It follows from \eqref{eq:109} that
$\dom(T_0^\go)=\iom\bigl(\dom(T)\bigr)$ and 
$T_0^\go(\iom x)= \iom ( T x ) $.\footnote{Originally we considered only the
localizations $T^\go$ constructed on the Hilbert space $E^\go$ (cf.
\eqref{eq:100}). We are indebted to Ryszard Nest for pointing out to us the
more general construction via the interior tensor product.}
\section{A separation theorem for Hilbert $C^*$--modules}\label{sepcstar} 
In this section we are going to prove the following separation theorem which  
will be the main tool for proving the Local--Global Principle, Theorem
\plref{t:locglob},  for regular operators. 

\begin{theorem}\label{t:sep} 
Let $L \subset E$ be a closed convex subset of the
Hilbert $C^*$--module $E$ over $\sA$. For each vector $x_0\in E\setminus L$ there
exists a state $\go$ on $\sA$ such that $\iom(x_0)$ is not in the closure of
$\iom(L)$. In particular there exists a state $\go$ such that $\iom(L)$ is
  not dense in $E^\go$ and thus, when $L$ is a submodule, $\iom(L)^\perp\not=\{0\}$.
\end{theorem}

\begin{remark} 
We emphasize that even if $L$ is a submodule it is not necessarily complementable.
If it is complementable then the statement
of the Theorem is obvious. Namely, write $x_0=x_0'+x_0''$ with $x_0'\in L$ and
$x_0''\in L^\perp$. $x_0''\not=0$ since $x_0\not \in L$. Furthermore,
$\iom (x_0'') \in \iom(L)^\perp$ for any state and choosing $\go$ such
that $\go(\inn{x_o'',x_0''})\not=0$ we have $\iom (x_0'')\not =0$.
\end{remark}

We mention two more pathologies of Hilbert $C^*$--modules which underline
that the Theorem should not be viewed as obvious.
\subsection{$\iom(L)$ can be dense for faithful $\go$}  
Let $\sA=C[0,1]$, $E=C[0,1]$, 
$L=\bigsetdef{f\in C[0,1]}{f(0)=0}$. $L$ is a closed non--trivial
submodule of $E$. The Lebesgue state $\go(f)=\int_0^1 f(t) dt$
is faithful, $E^\go\simeq L^2[0,1]$ and $\iom(L)$ is dense in $E^\go$.
So even for faithful states, and hence for faithful representations,
it may happen that $\iom(L)^\perp=\{0\}$.
\subsection{Convex hulls of closed subsets of $\sA_+\setminus\{0\}$ may
contain $0$}\label{ss:CHC}  
The proof of Theorem \plref{t:sep} will proceed by applying
the Hahn--Banach Theorem to the convex hull of the set
\begin{equation}\label{eq:102}  
        A:=\bigsetdef{\scalar{y-x_0}{y-x_0}}{ y\in L } \subset \sA_+.
\end{equation}
The closedness of $L$ implies that 
$\inf\bigsetdef{\|a\|}{a\in A} >0.$  
It will be crucial to show that the closure of the convex hull does not contain 
$0$.

We illustrate by example that in general we cannot hope that if $A\subset \sA_+$
with $\inf\bigsetdef{\|a\|}{a\in A}>0$ that then $0\not\in \co{A}$. Namely,
we will construct a subset $A\subset \sA=C[0,1]$ such that
\begin{itemize}
\item $A\subset \sA_+$,
\item $\|a\|\ge 1$ for all $a\in A$,
\item $0\in\co{A}$.
\end{itemize}

For $0<t<1$ and $n\in\Z_+$ such that $0<t-1/n, t+1/n<1$ let (cf. Figure
\ref{fig:one})
\begin{equation}\label{eq:113}  
    f_{t,n}(x):= \begin{cases} 0, & |x - t| \ge 1/n, \\
                               1- n |x-t|, & | x-t | \le 1/n.
                 \end{cases}			       
\end{equation}

\begin{figure}
\begin{tikzpicture}[scale=4]
\draw[->] (-0.2,0) -- (1.2,0);
\draw[->] (0,-0.2) -- (0,1.2);

\draw     (1,0.05) -- (1,-0.05)     node[anchor=north]{1};
\draw     (0.2,0.05) -- (0.2,-0.05) node[anchor=north]{$t-\frac 1n$};
\draw     (0.8,0.05) -- (0.8,-0.05) node[anchor=north]{$t+\frac 1n$};
\draw     (0.5,0.05) -- (0.5,-0.05) node[anchor=north]{$t$};
\draw[dotted]     (0,1) -- (1,1);
\draw[thick] (0,0) -- (0.2,0);
\draw[thick] (0.2,0) -- (0.5,1);
\draw[thick] (0.5,1) -- (0.8,0);
\draw[thick] (0.8,0) -- (1,0);

\node at   (1.2,0.5) {$f_{t,n}$};
\node at   (-0.2,1)  {$1$};
\end{tikzpicture}
\caption{The graph of the function $f_{t,n}$.}\label{fig:one}
\end{figure}
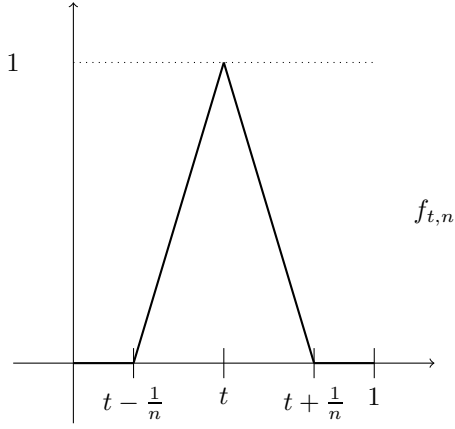

Let $A=\bigsetdef{f_{t,n}}{(t,n)\in\Q\times \Z_+, 0<t-1/n<t<t+1/n<1}$. Then $A$ is
a countable subset of $\sA_+$ and 
\begin{equation}\label{eq:114}  
    \inf\bigsetdef{\|f\|}{f\in A}=1.
\end{equation}

Now let $\eps>0$ be given. Choose a natural number $N>1/\eps$ and 
put $t_j:=j/(N+1), j=1,\ldots,N, n:=2N+2$. Then for $x\in[0,1]$
there is at most one index $j$ with $f_{t_j,n}(x)\not =0$. Hence for the
convex combination $\frac 1N \sum\limits_{j=1}^N f_{t_j,n}$ we have
\begin{equation}\label{eq:115}  
              0 \le \frac 1N \sum\limits_{j=1}^N f_{t_j,n}(x) \le \frac 1N<\eps     
\end{equation}
showing that $0\in \co{A}$, cf. Figure \plref{fig:two}.

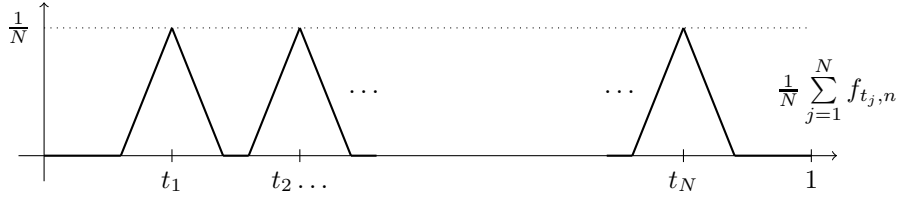
\begin{figure}
\begin{tikzpicture}[scale=1.7]
\draw[->] (-0.2,0) -- (6.2,0);
\draw[->] (0,-0.2) -- (0,1.2);

\draw     (1,0.05) -- (1,-0.05)     node[anchor=north]{$t_1$};
\draw     (2,0.05) -- (2,-0.05)     node[anchor=north]{$t_2\ldots$};
\draw     (5,0.05) -- (5,-0.05)     node[anchor=north]{$t_N$};
\draw     (6,0.05) -- (6,-0.05)     node[anchor=north]{$1$};
\draw[dotted]     (0,1) -- (6,1);
\draw[thick] (0,0) -- (0.6,0);
\draw[thick] (0.6,0) -- (1,1);
\draw[thick] (1,1) -- (1.4,0);
\draw[thick] (1.4,0) -- (1.6,0);
\draw[thick] (1.6,0) -- (2,1);
\draw[thick] (2,1) -- (2.4,0);
\draw[thick] (2.4,0) -- (2.6,0);
\node at (2.5,0.5) {$\mathbf{\ldots}$};
\node at (4.5,0.5) {$\mathbf{\ldots}$};
\draw[thick] (4.4,0) -- (4.6,0);
\draw[thick] (4.6,0) -- (5,1);
\draw[thick] (5,1) -- (5.4,0);
\draw[thick] (5.4,0) -- (6,0);

\node at   (-0.2,1)  {$\frac 1N$};
\node at   (6.2,0.5) {$\frac 1N \sum\limits_{j=1}^N f_{t_j,n}$};
\end{tikzpicture}
\caption{Convex combination of $f_{t_j,n}$ with arbitrarily small norm.}\label{fig:two}

\end{figure}
\subsection{Counterexample for pure states}\label{ss:CPS}
The previous construction  
can also be used to show that in Theorem
\plref{t:sep} ``state'' cannot be replaced by ``pure state''.
Namely, let $\sA=E=C[0,1]$ and let $L$ be the closed convex
hull of the two functions $f_{1/4,5}, f_{3/4,5}$. Then certainly
for each $f\in L$ we have $\|f\|\ge 1/2$,
hence $x_0=0\not\in L$.

Now let $\go$ be a pure state of $\sA$. Then there is $p\in [0,1]$
such that $\go(f)=f(p)$. Let $1\ge\eps>0$ be given. If $p\le 1/2$
then for $f=\eps f_{1/4,5}+(1-\eps)f_{3/4,5}$ we have $\go(\inn{f,f})\le\eps^2$.
If $p\ge 1/2$ then put $f=(1-\eps) f_{1/4,5}+\eps f_{3/4,5}$.
This argument shows that $0=\iom(x_0)$ is in the closure if $\iom(L)$.

\commentary{This example does not really convince me. $L$ is
just a convex set and not a submodule!}
\subsection{Proof of Theorem \plref{t:sep}}
Let now $A$ be the set defined in Eq. \eqref{eq:102}.
Since $L$ is closed we have
\begin{equation}\label{eq:101}  
        \delta:= \inf\bigsetdef{\|y-x_0\|^2}{y\in L} =\inf\bigsetdef{\|a\| }{a\in A} >0.    
\end{equation}
To apply the Hahn--Banach Theorem we need to show that $0\not\in\co{A}$.

To this end we consider arbitrary $y_1,\ldots, y_n\in E$ and real numbers
$\gl_j\ge 0$ with $\gl_1+\ldots+\gl_n=1$. Then (cf. \cite[Lemmas 4.2 and 4.3]{Lan:HCM})
\begin{equation}\label{eq:MatrixInequality}
\begin{split}
    \sum_{k,l=1}^n \gl_k \gl_l \scalar{y_k}{y_l} 
                    & = \sum_{k=1}^n \la_k^2 \inn{y_k,y_k}
                         + \sum_{k < l} \la_k \la_l \big( \inn{y_k,y_l} + \inn{y_l,y_k} \big) \\
& \leq \sum_{k=1}^n \la_k^2 \inn{y_k,y_k}
+ \sum_{k < l} \la_k \la_l \big( \inn{y_k,y_k} + \inn{y_l,y_l} \big) \\
& = \sum_{k=1}^n \la_k \inn{y_k,y_k}.
\end{split}
\end{equation}
Here we have used 
\begin{equation}\label{eq:149}  
\inn{x,y}+\inn{y,x}\le \inn{x,x}+\inn{y,y}
\end{equation}
which can be seen by expanding $\inn{x-y,x-y}\ge 0$.

Consider the convex combination
$\sum\limits_{j=1}^n \gl_j \scalar{y_j-x_0}{y_j-x_0}, y_1,\ldots, y_n\in L$, of elements of $A$.
Using \eqref{eq:MatrixInequality} we find
\begin{equation}\label{eq:106}  
 \begin{split}
  \sum_{j=1}^n &\gl_j \scalar{y_j-x_0}{y_j-x_0}\\
  & = \scalar{x_0}{x_0}-\sum_{j=1}^n \gl_j\bigl(
  \scalar{y_j}{x_0}+\scalar{x_0}{y_j}\bigr)+\sum_{j=1}^n \gl_j \scalar{y_j}{y_j}\\
  & \ge \scalar{x_0}{x_0}-\sum_{j=1}^n \gl_j\bigl(
  \scalar{y_j}{x_0}+\scalar{x_0}{y_j}\bigr)+\sum_{k,l=1}^n \gl_k \gl_l
  \scalar{y_k}{y_l}\\
  &= \scalar{x_0-\sum_{j=1}^n \gl_j y_j}{x_0-\sum_{j=1}^n \gl_j y_j}.
 \end{split}   
\end{equation}
Since $L$ is assumed to be convex, $\sum \gl_j y_j\in L$, hence \eqref{eq:101} and \eqref{eq:106} give
\begin{equation}\label{eq:107}  
  \Bigl\|\sum_{j=1}^n \gl_j \scalar{y_j-x_0}{y_j-x_0}\Bigr\| \ge \delta.
\end{equation}
This shows that each element $b$ in the closure of the convex hull
$\overline{\textup{co}(A)}$ of $A$ satisfies $\|b\|\ge \delta$. This proves
that $0 \notin \co{A}$.

The Hahn--Banach separation theorem now implies the existence of a continuous
linear functional $\varphi : \sA_{\T{sa}} \to \R$ and an $\eps> 0$ such that
$\varphi(b) > \eps$ for all $b \in \ov{\T{co}(A)}$. Here $\sA_{\T{sa}}$ denotes
the real Banach space of selfadjoint elements in the $C^*$--algebra $\sA$. We
extend the linear functional $\varphi$ to a selfadjoint linear functional on
the $C^*$-algebra $\sA$ by defining
\[
\tau : \sA \to \cc,  \quad \tau(x) := \varphi\big( \frac{x+x^*}{2}\big)
+ i \varphi \big( \frac{x - x^*}{2i}\big).
\]
By Jordan decomposition for $C^*$--algebras we can then find two positive
linear functionals $\go_\pm\in \sA_+^*$ such that $\tau =\go_+-\go_-$. Hence
$\go_+(b) \geq \varphi(b) > \eps$ for all $b \in \ov{\T{co}(A)} \su
\sA_+$. Putting $\go=\go_+/\|\go_+\|$ we see that in $E^\go$ the vector
$\iom(x_0)$ and the subspace $\iom(L)$ have distance at least
$\sqrt{\eps/\|\go_+\|}>0$ which
proves the claim. \hfill\qed 

\subsection{Application: A core--criterion for semiregular
operators}\label{ss:CCS} 
\begin{theorem}\label{t:CoreCriterion}
Suppose that $T$ is a closed and semiregular operator in the Hilbert
$\sA$--module $E$. Let $\sE \su \dom(T)$ be a submodule of the domain of
$T$. The following statements are then equivalent:
\begin{thmenum}
\item The submodule $\sE$ is a core for $\dom(T)$.
\item For every representation $(\pi,H_\pi)$ of $\sA$ the subspace
$\sE\ot_\sA H_\pi$ is a core for $T^\pi$. 
\item For every state $\go\in S(\sA)$ the subspace 
$\iom(\sE) \su \dom(T^\om)$ is a core for the localization  $T^\om$.
\end{thmenum}
\end{theorem}
\begin{proof} Firstly, the implication $(2)\Rightarrow (3)$ is clear.
Secondly, we note that for any representation
$(\pi, H_\pi)$ the scalar product on $\dom(T)\ot_\sA H_\pi$ (induced by the
graph scalar product on $\dom(T)$) equals the graph scalar product of $T^\pi$.
Namely, for $x,y\in\dom(T), h,h'\in H_\pi$ we have
\begin{equation}\label{eq:117}  
\begin{split}
      \inn{x\ot h, &y\ot h'}_{\dom(T)\ot_\sA H_\pi} 
          = \inn{h, \pi(\inn{x,y}_T) h'}\\    
         & = \inn{h, \pi(\inn{x,y}) h'}+ \inn{h, \pi(\inn{Tx,Ty}) h'} \\
	 & = \inn{x\ot h, y\ot h}_{E\ot_\sA H_\pi} + \inn{T^\pi(x\ot h), T^\pi(y\ot h)}_{E\ot_\sA H_\pi}\\
	 & = \inn{x\ot h, y\ot h'}_{T^\pi}.
\end{split}   
\end{equation}
This shows that $\dom(T)\hot_\sA H_\pi=\dom(T^\pi)$ as Hilbert spaces.

In light of this if $\sE$ is dense in $\dom(T)$ then so
is $\sE \ot_\sA H_\pi$ in $\dom(T^\pi)$ proving $(1)\Rightarrow (2)$.

\subsubsection*{$ \neg (1)\Rightarrow \neg (3)$}  If $\sE$ is not a core for $T$ then
there exists a vector $x_0\in \dom(T)\setminus \ov \sE$. Hence by Theorem \plref{t:sep}
there exists a state $\go$ such that $\iom^T(x_0)=x_0 \ot \xi_\go$ is not in the closure of
$\iom^T(\sE)$. Here $\iom^T$ denotes the natural map $\dom(T)\longrightarrow
\dom(T)^\go$. Thus $\iom(\sE)$ is not a core for $T^\go$.
\end{proof}
\section{The Local--Global Principle} \label{s:LGP} 
Before we prove the main theorem of this section we recall the  
characterization of selfadjoint regular operators in terms of the range of the
operators $T \pm i $, \cite[Lemmas 9.7 and 9.8]{Lan:HCM}:

\begin{prop}\label{charregself}
Let $T$ be a closed, densely defined and symmetric operator in
the Hilbert $C^*$--module $E$ over $\sA$. Then for $\mu\in \R\setminus \{0\}$ the operator  $T\pm i\mu$ 
is injective and has closed range. Furthermore, the following statements are then equivalent:
\begin{thmenum}
\item The unbounded operator $T$ is selfadjoint and regular.
\item There exists $\mu>0$ such that each of the operators $T + i\mu $ and $T - i\mu $ has dense range.
\end{thmenum}
\end{prop}
It then follows that $T\pm i\mu$ is invertible for all $\mu\in\R\setminus\{0\}$.
In \cite{Lan:HCM} (2) is stated for $\mu=1$. The slight extension to arbitrary nonzero
$\mu$ is proved as in the Hilbert space setting and left to the reader.

We remark that regularity is a consequence of selfadjointness when the Hilbert
$C^*$--module is a Hilbert space. This property and the separation theorem for
Hilbert $C^*$--modules proved in Section \ref{sepcstar} are applied in the
proof of the next theorem.

\begin{theorem}[Local--Global Principle]\label{t:locglob}\indent\par
\textup{1. } For a closed semiregular operator $T$ in a Hilbert $C^*$--module 
the following statements are equivalent:
\begin{thmenum}
\item $T$ is regular.
\item For every representation $(\pi,H_\pi)$ of $\sA$ the localizations
$T^\pi$ and $(T^*)^\pi$ are adjoints of each other, i.e. $(T^*)^\pi =
(T^\pi)^*$.
\item For every state $\go\in S(\sA)$ the localizations
$T^\go$ and $(T^*)^\go$ are adjoints of each other.
\end{thmenum}

\textup{2. } For a closed, densely defined and symmetric operator $T$
the following statements are equivalent:
\begin{thmenum}
\item $T$ is selfadjoint and regular.
\item For every representation $(\pi,H_\pi)$ of $\sA$ the localization $T^\pi$
is selfadjoint.
\item For every state $\go\in S(\sA)$ the localization $T^\go$ is selfadjoint.
\end{thmenum}
\end{theorem}
\begin{remark} We note that under 1.(2) the identity $(T^*)^\pi =(T^\pi)^*$
 implies $T^\pi=\bigl((T^*)^\pi\bigr)^*$ since $T^\pi$ is a closed operator in
 a Hilbert space and therefore $T^\pi = (T^\pi)^{**}$.
\end{remark} 
\begin{proof}
In light of Lemma \plref{l:THat} it suffices to prove 2. The implication $(2)\Rightarrow (3)$
is obvious. 

\subsubsection*{$(1) \Rightarrow (2)$}
Assume that $T$ is selfadjoint and regular and let $(\pi, H_\pi)$ be a
representation of $\sA$. By Proposition \ref{charregself} we only need to prove that $T^\pi
+ i$ and $T^\pi -i$ have dense range. W.l.o.g. consider $T^\pi + i$. 
Since $E\ot_\sA H_\pi$ is dense in $E^\pi$ and by linearity it suffices to
show that $x\ot h\in \ran(T^\pi+i)$ for $x\in E$ and $h\in H_\pi$. Since $T$ is
selfadjoint and regular $T + i$ is surjective and hence $y:=(T+i)\ii x\in\dom(T)$
exists. Then $(T^\pi+i)(y\ot h)=x\ot h$.

\subsubsection*{$(3)\Rightarrow (1)$}
Next we prove that the selfadjointness of all the localized operators imply
the selfadjointness and regularity of the global operator.

Thus assume that the localized operator $T^\om$ is selfadjoint for each state
$\om \in S(A)$. Assume by contradiction that the range
of $T + i$ is not dense in $E$. By Proposition \plref{charregself} the range $\ran(T + i)$ is a
proper closed submodule of $E$. By Theorem \ref{t:sep} there exists 
a state $\om \in S(\sA)$ such that
\begin{equation}\label{eq:153}  
\ov{i_\om\big( \ran(T+i) \big)} \neq E^\om.
\end{equation}
However, we also have the identities of subspaces
\[
\ov{i_\om\big( \ran(T+i) \big)} 
= \ov{\ran(T^\om_0 + i)}
= \ran(T^\om + i).
\]
Thus $T^\om + i$ does not have dense range which is in contradiction with the
selfadjointness of $T^\om$. The same argument shows that the operator $T - i$ has dense range as well
and the Theorem is proved.
\end{proof}

\begin{remark}
The PhD-thesis of Baaj \cite{Baa:MNB} seems to be the earliest detailed treatment of regular operators, though only for the special case where the Hilbert $C^*$-module is the $C^*$-algebra itself. This work contains the functional calculus as well as both of the implications $(1)\Rightarrow(2)$ in Theorem \ref{t:locglob}.
\end{remark}
\subsection{Application: W\"ust's extension of the Kato--Rellich Theorem} \label{ss:WKRT}   
The Kato--Rellich Theorem \cite[Theorem X.12]{ReeSim:MMMII} extends to
Hilbert $C^*$--modules without any difficulty.

\begin{theorem}[Kato--Rellich]
Let $T : \dom(T) \to E$ be a selfadjoint regular operator and let $V : \dom(V)
\to E$ be a symmetric operator such that $\dom(T) \su \dom(V)$. Suppose that $V$
is relatively $T$--bounded with relative bound $<1$. That is there exist $a\in (0,1), b\in \R_+$
such that for $x\in\dom(T)$
\[
     \|Vx\| \le a \| T x \| + b \| x \|.
\]
Then $T+V$ with domain $\dom(T)$ is selfadjoint and regular.
\end{theorem}
\begin{proof} The standard Hilbert space proof extends to this situation: namely,
for $\mu\in\R_+$ large enough the operators $T+V\pm i \mu=  (I+ V(T\pm
i\mu)\ii)(T\pm i \mu)$
is invertible and hence $T+V$ is selfadjoint and regular.
\end{proof}

The proof of W\"ust's extension to the case of relative bound $1$ \cite[Theorem X.14]{ReeSim:MMMII}
makes heavy use of the fact that Hilbert spaces are self--dual and of weak compactness
of the unit ball. These tools are not available for Hilbert $C^*$--modules. Our Local--Global
Principle allows us to generalize W\"ust's Theorem as follows:

\begin{theorem}[W\"ust]\label{t:Wust}
Let $T : \dom(T) \to E$ be a selfadjoint regular operator and let $V : \dom(V)
\to E$ be a symmetric operator such that $\dom(T) \su \dom(V)$. Suppose that
there exists a
$b\in \R_+$ such that for $x\in\dom(T)$
\[
    \inn{Vx,Vx}\le \inn{Tx,Tx}+b \inn{x,x}.
\]    
Then $T+V$ with domain $\dom(T)$ is
essentially selfadjoint and regular.
\end{theorem}
\begin{proof} 
Let $\go\in S(\sA)$ be a state of $\sA$. Then we have for $x\in \dom(T)$
\begin{equation}\label{eq:118}  
    \begin{split}
          \| V^\go_0 \iom (x))\|^2 & = \go(\inn{Vx,Vx}) \le \go(\inn{Tx,Tx})
          +b\cdot  \go (\inn{x,x})\\
	      & = \| T^\go_0 \iom (x)\|^2 + b\cdot \| \iom (x)\|^2,
    \end{split}
\end{equation}
thus $V^\go_0$ is relative $T_0^\go$--bounded with relative bound $1$. Taking closures
shows that $V^\go$ is relative $T^\go$--bounded with relative bound $1$, too. By W\"ust's
Theorem \cite[Theorem X.14]{ReeSim:MMMII} it follows that $T^\go+V^\go$ is essentially
selfadjoint on any core for $T^\go$. In particular it is essentially selfadjoint on $\iom(\dom(T))$.
For $\iom (x), x\in\dom(T)$, however, we have $(T^\go+V^\go )\iom (x)=
\iom(Tx+Vx)=(T+V)_0^\go \iom (x)$.
Thus the localization $(T+V)^\go$ of $T+V$ is selfadjoint.

The claim now follows from Theorem \plref{t:locglob}.
\end{proof}

\section{Pure states, commutative algebras and involutive Hilbert $C^*$--modules} \label{s:Pure}   
Section \plref{ss:CPS} shows that in Theorem \plref{t:sep} one cannot conclude 
that $\go$ can be chosen to be pure. 

\begin{dfn}\label{def:AConvex}
 \textup{1. } $\{\varrho_j\}_{j=1}^n$ is called a partition
 of unity if 
 \begin{thmenum}
 \item $\varrho_j\in \sA, $\quad $j=1,\ldots, n-1$ and $\varrho_n\in \sA^+$,
 \item $\sum\limits_{j=1}^n \varrho_j^*\varrho_j=I.$
 \end{thmenum} 
 Here $\sA^+$ is $\sA$ if $\sA$ is unital and otherwise it denotes
 the unitalization of $\sA$; $I$ is the unit in $\sA^+$.

 \textup{2. } A subset $A\subset \sA$
is called $\sA$--convex if for any $x_1,\ldots,x_n\in A$ and
a partition of unity $\varrho_j\in \sA, j=1\ldots,n$ one has
\[
   \sum_{j=1}^{n} \varrho_j^*\,  x_j\, \varrho_j\, \in A.
\]
\end{dfn}

\begin{conjecture} \label{c:one} If in the situation of Theorem \plref{t:sep}
 $L$ is an $\sA$--submodule then there exists a pure state $\go$ such that 
$\iom(x_0)$ is not in the closure of
$\iom(L)$. In particular there exists a pure state $\go$ such that $\iom(L)$ is
  not dense in $E^\go$ and hence $\iom(L)^\perp\not=\{0\}$. 
\end{conjecture} 

\begin{conjecture} \label{c:two} Let $\sA$ be a $C^*$--algebra and
 let $A\subset \sA_+$ be a closed $\sA$--convex subset of the positive cone of $\sA$.
 If $0\not\in A$ then there exist an $\eps>0$ and a pure state $\go$ such that
 $\go(a)\ge \eps$ for all $a\in A$.
\end{conjecture}

\begin{conjecture}\label{c:three} Conjecture \plref{c:two} implies 
 Conjecture \plref{c:one}.
\end{conjecture}

\begin{remark} Theorem \plref{t:sep} and its proof show that if one replaces
 ``$\sA$--convex'' by the weaker condition ``convex'' then all three conjectures
 hold true if one replaces ``pure state'' by the weaker conclusion 
 ``state''. Section \plref{ss:CPS} shows that under the weaker condition
 ``convex'' the statement of Conjecture \plref{c:two} becomes false
 for pure states.
\end{remark}

\begin{theorem}\label{t:LocGlobPure} If Conjecture \plref{c:one} holds
 for a $C^*$--algebra $\sA$ then in statement \textup{(3)} under \textup{1.}
 and \textup{2.} of  Theorem
 \plref{t:locglob}  ``state'' can be replaced by ``pure state''.
\end{theorem} 
\begin{remark}\label{rem:LocGlobPure}
Actually, for this conclusion to hold the second sentence in
Conjecture \plref{c:one} suffices.
\end{remark}
\begin{proof} One argues as in the proof of Theorem \plref{t:locglob}
 replacing the conclusion of Theorem \plref{t:sep} by that of
 Conjecture \plref{c:one}. With regard to the previous remark we
 emphasize that indeed in \eqref{eq:153} only the last sentence
 of Theorem \plref{t:sep} was used.
\end{proof} 
\subsection{Commutative algebras}\label{ss:CA}
We are now going to prove that all three conjectures are true
for commutative $C^*$--algebras.

\begin{theorem}\label{t:Pure} If $\sA$ is commutative then Conjectures
 \plref{c:one}, \plref{c:two}, and \plref{c:three} hold.
\end{theorem} 

\begin{proof}
\indent\par
\subsubsection*{ \plref{c:two} $\Rightarrow$ \plref{c:one} }
The inequalities \eqref{eq:MatrixInequality} and  \eqref{eq:106}
are proved verbatim for $\gl_j=\varrho_j^* \varrho_j$ and $y_j\in L$.
Since $\sA$ is commutative they can be checked pointwise
on the Gelfand spectrum of $\sA$. Thus as in the proof of
Theorem \plref{t:sep} one concludes that the $\sA$--convex
hull $\widetilde A$ of the set $A$ defined in \eqref{eq:102} 
does not have $0$ in its closure. Conjecture \plref{c:two} now gives
us a pure state $\go$ which implies the validity of Conjecture \plref{c:one}.

\subsubsection*{ We will now prove Conjecture \plref{c:two} for $\sA$
commutative}
Let $X$ be the Gelfand spectrum of $\sA^+$. This is a compact Hausdorff
space. If $\sA$ is unital then $\sA\simeq C(X)$ and if $\sA$ is non-unital 
then there is a distinguished point $\infty\in X$ such that
$\sA\simeq \bigsetdef{f\in C(X)}{ f(\infty)=0}$.

Furthermore, each $p\in X$ ($X\setminus\{\infty\}$ in the non-unital case)
gives rise to a pure state $\go_p(f)=f(p)$ and every pure state
arises in this way.

We proceed by contradiction and assume that the conclusion of Conjecture
\plref{c:two} does not hold. Then for given $\eps>0$ and each
$p\in X$ there exist an open neighborhood $U_p$ of $p$ and an $f_p\in A$ with 
$f_p(q)\le \eps$ for $q\in U_p$. By compactness there exist finitely
many $p_1,\ldots, p_n$ such that $X=U_{p_1}\cup \ldots \cup U_{p_n}$.
In the non-unital case we may choose and enumerate them such that $\infty\in U_{p_n}$
and $\infty\not\in U_{p_j}$ for $j=1,\ldots,n-1$.
Since compact spaces are paracompact there exists a subordinated
partition of unity $\chi_1,\ldots,\chi_n, \chi_j\in C_c(U_{p_j})$.
The set $\{\sqrt{\chi_j}\}_{j=1}^n$ is then
a partition of unity in the sense of Definition \plref{def:AConvex}.
Hence $f:=\sum\limits_{j=1}^n \chi_j f_{p_j}=\sum\limits_{j=1}^n \sqrt{\chi_j}\, f_{p_j}
\, \sqrt{\chi_j}$ 
is in $A$ by $\sA$--convexity and it satisfies
$0\le f \le \eps$. This shows that $0$ is in $A$ contradicting the
assumption $0\not\in A$.
\end{proof} 

We can now easily deduce the following generalization of a 
result of Pal \cite[Prop. 4.1]{Pal:ROH} about regular operators
over commutative $C^*$--algebras.

\begin{theorem}\label{t:ROvectorbundles} Let $E$ be a \emph{finitely}
generated Hilbert $C^*$--module over the commutative $C^*$--algebra $\sA$.
Then every semiregular operator $T$ in $E$ is regular.
\end{theorem}
\begin{proof}
As before let $X$ be the Gelfand spectrum of $\sA^+$.
By Theorem \plref{t:Pure} it suffices to show that for each pure
state $\go_p, p\in X,$ the localized operator $T^\go$ satisfies
$(T^*)^\go=(T^\go)^*$. 

Let $f_1,\ldots,f_N\in E$ be a generating set over $\sA$.
Recall that $E^\go$ is the Hilbert space completion of 
$E/\sN_\go$ with respect to the scalar product $\inn{f,g}(p)$.
Given a fixed $f\in E$. Then for $\varphi\in\sA$ the vector
$\iom(f\varphi)$ depends only on the value $\varphi(p)$ and
hence $E^\go$ is the vector space spanned by $\iom(f_1),\ldots,\iom(f_N)$
and thus is finite--dimensional. $T^\go$ and $(T^*)^\go$ are,
as densely defined closed operators in a finite--dimensional vector space,
everywhere defined and bounded. The inclusion $(T^*)^\go\subset (T^\go)^*$
(Lemma \plref{l:LocalizedOperator}) then implies $(T^*)^\go = (T^\go)^*$.
\end{proof}

\subsection{Regular operators over $C^*$--algebras and involutive Hilbert
$C^*$--modules}\label{ss:ROCA} 

\begin{theorem}\label{t:InvHMPureState}
Let $\sA$ be a $C^*$--algebra. Then for the Hilbert $C^*$--module
$E=\sA$ the conclusion of the last sentence in Conjecture \plref{c:one}
holds and hence in statement \textup{(3)} under \textup{1.}
 and \textup{2.} of  Theorem
 \plref{t:locglob}  ``state'' can be replaced by ``pure state''.
\end{theorem}
\begin{remark}\label{rem:InvHMPureState}
1. So for unbounded semiregular operators over $C^*$--algebras 
regularity can be checked by looking at
the localizations with respect to pure states. We emphasize that
\cite[Sec. 3]{Pal:ROH} does not help in extending this result
to semiregular operators over general Hilbert $C^*$--modules
because loc. cit. only shows that a (semi)regular operator
in $E$ is equivalent to an operator in $\sK(E)$.
Our Theorem then says that one could check regularity now
by looking at the localizations of the equivalent operator
with respect to the pure states of $\sK(E)$.

2. Theorem \plref{t:InvHMPureState} can in principle be
extracted from \cite[Prop. 2.5]{Wor:UEA}, although there
it is not stated explicitly. Our proof, however, is basically
the same as the one in loc. cit.
\end{remark}
\begin{proof}
Let $L$ be a proper $\sA$ right submodule of $E$. Then $L$
is a proper \emph{right ideal} in $\sA$ and hence
by \cite[Thm. 2.9.5]{Dix:CA} there exists a pure state
$\go\in S(\sA)$ such that $\go\restriction L=0$.
Hence $\iom(L)=\{0\}$ but since $\go$ is a state
certainly $E^\go\not=\{0\}$ proving $\iom(L)^\perp\not=\{0\}$.
\end{proof}

The argument of the previous proof exploits that the Hilbert $C^*$--module
$E=\sA$ has a little more structure:
Namely, it is a left-- and a right module. The inner product compatible
with the right module structure is $\inn{a,b}=a^*b$ and the
inner product compatible with the left module structure is $\inn{a,b}_l=ab^*$.
Obviously, the involution $*:\sA\to\sA$ has the property
$\inn{a,b}=\inn{a^*,b^*}_l$. This motivates

\begin{dfn}\label{def:InvHM}
An involutive Hilbert $C^*$--module over a $C^*$--algebra $\sA$ is a
Hilbert $C^*$--module $E$ together with a bounded involution
$*:E\to E$.
\end{dfn}

Putting $a x:=(x a^*)^*$ and $\inn{x,y}_l:=\inn{x^*,y^*}$
for  $a\in\sA$, $x,y\in E$ gives $E$ the structure of an $\sA$ Bimodule
such that $(E,\inn{\cdot,\cdot}_l)$ is a left Hilbert $\sA$--module
and $(E,\inn{\cdot,\cdot})$ is a right Hilbert $\sA$--module.

For any $C^*$--algebra $\sA$ the space $E=\sA^n$ is an involutive
Hilbert module via $(a_j)^*_j:=(a_j^*)_j$. However, for noncommutative $\sA$
the countably generated Hilbert module $H_\sA$ is not necessarily involutive
since $(a_j)_j\mapsto (a_j^*)_j$ is not necessarily bounded.
(Except for the trivial commutative case) we do not know of other
interesting involutive Hilbert $C^*$--modules.

We believe that Theorem \plref{t:InvHMPureState} extends to
full involutive Hilbert $C^*$--modules. But in the lack of good
examples we do not follow this path any further and leave
the details to the reader.
\newcommand{\Dmax}{D_{\max}}
\newcommand{\Dmin}{D_{\min}}
\newcommand{\Dmaxmin}{D_{\maxmin}}
\newcommand{\Tmax}{T_{\max}}
\newcommand{\Tmin}{T_{\min}}
\newcommand{\Tmaxmin}{T_{\maxmin}}
\newcommand{\TLa}{T_{\gL}}
\newcommand{\regL}{\reg(\gL)}
\newcommand{\regiL}{\reg_\infty(\gL)}
\newcommand{\sisuL}{\ssupp(\gL)}
\newcommand{\sisurL}{\ssupp_r(\gL)}
\section{Examples of non--regular operators} \label{s:NRO}  
In this section we will recast in a slightly more general context
the known constructions of nonregular operators \cite[Chap. 9]{Lan:HCM}, \cite{Pal:ROH}.

Fix a separable Hilbert space and a symmetric closed operator $D$ with deficiency
indices $(1,1)$. E.g.  $H=L^2[0,1]$, 
\[\dom(D):=H^1_0[0,1]=\bigsetdef{f\in L^2[0,1]}{f'\in L^2[0,1], \, f(0)=f(1)=0}\]
and $Df:=-i f'$ will do.

As usual we put $\Dmin:=D, \Dmax:=D^*$. The domain $\dom(\Dmax)$ is a Hilbert space in its
own right with respect to the graph scalar product (cf. \eqref{eq:116}) and there
are two normalized vectors $\phi_\pm$ such that
\begin{equation}\label{eq:119}  
\begin{split}
          \Dmax \phi_\pm &= \pm\, i \phi_\pm,  \\
	  \dom(\Dmax)    &= \dom(\Dmin)\oplus \C\, \phi_+\oplus \C \, \phi_-,
\end{split}	  
\end{equation}
where the orthogonal sum and the normalization of $\phi_\pm$ are understood
with respect to the graph scalar product of $\Dmax$. Therefore,  
\begin{equation}\label{eq:130}   
         1= \| \phi_\pm\|_D^2 = \|\phi_\pm\|^2 + \|\pm i \phi_\pm\|^2 = 2
         \|\phi_\pm\|^2,\quad 
         \|\phi_\pm\| = \frac{1}{\sqrt{2}}.
\end{equation}
We introduce two continuous linear functionals
\begin{equation}\label{eq:120}  
   \ga_\pm:\dom(\Dmax)\longrightarrow \C,\quad \ga_\pm(\xi):=\inn{\phi_\pm,\xi}_{D}. 
\end{equation}
With these notations the selfadjoint extensions of $D$ are parametrized by $\gl\in S^1$: 
for $\gl\in S^1$ the operator $D_\gl$ is $\Dmax$ restricted to
\begin{equation}\label{eq:121}  
    \dom(D_\gl)=\bigsetdef{\xi\in\dom(\Dmax)}{ \ga_+(\xi)= \gl \, \ga_-(\xi)}.
\end{equation}
With
\begin{equation}\label{eq:131}  
 \eta_\gl:=\frac{1}{\sqrt{2}}\bigl( \gl \phi_+ + \phi_-\bigr), \quad 
 \eta_\gl^\perp:=\frac{1}{\sqrt{2}}\bigl( \phi_+ - \ovl{\gl} \phi_-\bigr),  
\end{equation}
we have
\begin{equation}\label{eq:132}  
	  \dom(\Dmax)    = \dom(\Dmin)\oplus \C\, \eta_\gl\oplus \C \,
   \eta_\gl^\perp  = \dom(D_\gl) \oplus \C \, \eta_\gl^\perp,
\end{equation}
hence $\xi\in \dom(\Dmax)$ lies in $\dom(D_\gl)$ iff $\xi\perp \eta_\gl^\perp$
with respect to the graph scalar product, equivalently if
\begin{equation}\label{eq:133}  
            \inn{\xi, \Dmax(\gl\phi_++\phi_-)}=\inn{\Dmax \xi, \gl \phi_++\phi_-}.    
\end{equation}

Next let $X$ be a locally compact Hausdorff space, $C_0(X)$ the $C^*$--algebra of continuous functions which vanish at infinity.
$E:=C_0(X,H)=H\hot_{C_0(X)} C_0(X)$ is the standard Hilbert $C^*$--module over $C_0(X)$ modeled on $H$
with inner product
\begin{equation}\label{eq:122}  
 \inn{f,g}(x):= \inn{f(x),g(x)}_H.
\end{equation}
We are now going to introduce semiregular operators $\Tmin, \Tmax$ as
follows: let $\dom(\Tmaxmin):=C_0(X,\dom(\Dmaxmin))=\dom(\Dmaxmin)\hot_{C_0(X)}
C_0(X)$. These are Hilbert $C^*$--modules and we have natural continuous
inclusions
\begin{equation}\label{eq:123}  
    \dom(\Tmin)\hookrightarrow \dom(\Tmax)\hookrightarrow E.
\end{equation}
For $f\in \dom(\Tmaxmin)$ put $(\Tmaxmin f)(x):=\Dmaxmin\bigl(f(x)\bigr)$.

\begin{lemma}\label{l:TmaxminRegular} $\Tmaxmin$ are closed regular operators in $E$ with
 $\Tmax^*=\Tmin,\, \Tmin^*=\Tmax$. Furthermore, the natural inner
 products on $C_0(X,\dom(\Dmaxmin))$ coincide with the graph inner products
 of $\Tmaxmin$. Hence by Prop. \plref{p:GraphHilbertModule} the inclusion maps \eqref{eq:123}
 are adjointable.
\end{lemma}
\begin{proof} Certainly for $f,g\in C_0(X,\dom(\Dmax))$ we have
\begin{equation}\label{eq:127}  
 \begin{split}
 \inn{f,g&}_{C_0(X,\dom(\Dmax))}(x)= \inn{f(x),g(x)}_{\dom(\Dmax)}\\
          & =\inn{f(x),g(x)}_H+ \inn{\Dmax (f(x)), \Dmax (g(x))}_H \\
          & = \inn{f,g}_E(x) + \inn{\Tmax f, \Tmax g}_E(x)\\
          & = \inn{f,g}_{\Tmax}(x),
         \end{split}          
\end{equation}
proving the claim about graph inner products. Since $C_0(X,\dom(\Dmaxmin))$ are
Hilbert $C^*$--modules this also shows that $\Tmaxmin$ are closed operators.

If $f\in\dom(\Tmin^*)$ then for each $g\in\dom(\Tmin)=C_0(X,\dom(\Dmin))$ and
each $x\in X$
\begin{equation}\label{eq:128}  
 \inn{(\Tmin^* f)(x),g(x)}= \inn{f(x),\Dmin\bigl( g(x)\bigr)},
\end{equation}
hence $f(x)\in\dom(\Dmax)$ and $\Dmax\bigl(f(x)\bigr)=(\Tmin^*f)(x)$. This
proves that $f\in\dom(\Tmax)$ and $\Tmax f=\Tmin^*f$. This argument
proves $\Tmin^*\subset \Tmax$. The inclusion $\Tmax\subset\Tmin^*$ is obvious and hence
we have equality. The equality $\Tmax^*=\Tmin$ now follows similarly.

To prove regularity we only have to note that for given $f\in E$
the elements defined by $g_1(x):=(I+\Dmax \Dmin)^{-1}f(x)$ resp.
$g_2(x):=(I+\Dmin \Dmax)^{-1}f(x)$ are in $E$ and even more
lie in the domain of $\Tmax\Tmin$ resp. $\Tmin\Tmax$ and
$(I+\Tmax\Tmin)g_1=f$ resp. $(I+\Tmin\Tmax)g_1=f$.
\end{proof}

We are now going to study semiregular extensions $\Tmin\subset \TLa \subset \Tmax$ 
which depend on a Borel function
$\gL:X\longrightarrow S^1$. For a Borel function $\gL$ we let $\TLa$
be the operator $\Tmax$ restricted to
\begin{equation}\label{eq:124}  
 \dom(\TLa):=\bigsetdef{f\in\dom(\Tmax)}{ \ga_+\circ f = \Lambda \cdot
 (\ga_-\circ f)}.
\end{equation}
$\TLa$ is a closed operator, $\Tmin\subset\TLa\subset \Tmax$, hence
$\TLa^*\supset \Tmax^*=\Tmin$. Thus $\TLa$ is a semiregular operator.
In view of Theorem \plref{t:Pure} for characterizing the regularity of $\TLa$
it suffices to study its localizations $\TLa^p$ with respect to the points
$p\in X$ (i.e. the pure states on $C_0(X)$). As a preparation we define the
following subsets of $X$ depending on the function $\Lambda$:
\begin{align}
      \regL&:=\bigsetdef{p\in X}{\gL \text{ continuous in a neighborhood
      of } p},\\
      \sisuL&:= X\setminus \reg(\gL).
\end{align}      
$\regL$ is the largest open subset of $X$ on which $\gL$ is
continuous, $\sisuL$ is the closed singular support of $\gL$.
We furthermore distinguish two kinds of points in the common boundary 
$\pl \regL = \pl \sisuL$. For a point in $\pl \regL$ we say
that $p\in \regiL$ if there exists an open neighborhood $U$ of
$p$ and a continuous function $\tilde \gL:U\to S^1$ such that
\begin{equation}\label{eq:134}  
    \tilde \gL \restriction U\cap \regL = 
    \gL \restriction U\cap \regL.
\end{equation}
For these $p$ the limit
\begin{equation}\label{eq:125}  
 \tilde\Lambda(p):=\lim\limits_{q\to p,\; q\in \regL} \Lambda(q)\in S^1
\end{equation}
exists and hence the value $\tilde\gL(p)\in S^1$ is uniquely determined. 
However, the existence of the limit \eqref{eq:125} does in general not imply
that $p\in\regiL$.
Finally, 
\begin{equation}\label{eq:135}  
    \sisurL:=\bigl(\pl\sisuL\bigr)\setminus \regiL
\end{equation}
denotes the complement of $\regiL$ in $\pl\sisuL$.

\begin{lemma}\label{l:TLambdap}
 The localizations of $\TLa$ and $\TLa^*$ with respect to pure states 
 are given as follows:
\begin{equation}\label{eq:126}  
 \TLa^p:=\begin{cases} \Dmin, & p\in\ssupp \gL,\\
                       D_{\gL(p)}, & p\in \reg \gL.
         \end{cases}
\end{equation}
\begin{equation}\label{eq:126a}  
 (\TLa^*)^p:=\begin{cases} \Dmax, & p\in(\sisuL)^\circ,\\
                       D_{\gL(p)}, & p\in \regL,\\
                       D_{\tilde \gL(p)}, & p\in \reg_\infty \gL,\\
                       \Dmin, & p\in \ssupp_r \gL.
         \end{cases}
\end{equation}
\end{lemma}
\begin{proof}
We first note that $\dom(\Dmin)\subset\dom(T_{\gL,0}^p) \subset
\dom(D_{\gL(p)})$. The second inclusion follows from the definition
\eqref{eq:Tpi0}. To see the first inclusion let $\xi\in\dom(\Dmin)$. Then
the constant function $f(x):=\xi$ lies in $\dom(\Tmin)\subset\dom(\TLa)$
and $f(p)=\xi$. Since $\dom(D_{\gL(p)})/\dom(\Dmin)$ is one--dimensional
it follows that either $\dom(T_{\gL,0}^p)=\dom(\TLa^p)=\dom(\Dmin)$
or  $\dom(T_{\gL,0}^p)=\dom(\TLa^p)=\dom(D_{\gL(p)})$. Suppose
that $f\in\dom(\TLa)$ with $\ga_-(f(p))\not=0$. Then there exists
an open neighborhood $U$ of $p$ such that $\ga_-(f(q))\not=0$ for $q\in U$
and hence 
\begin{equation}\label{eq:136}  
 \Lambda\restriction U=\frac{\ga_+ \circ f}{\ga_-\circ f}\restriction U
\end{equation}
is continuous on $U$, proving $p\in \regL$. Thus if $p\not\in\regL$
we have $\dom(\TLa^p)=\dom(\Dmin)$. Continuing with $p\in\regL$ choose
a function $\varphi\in C_c(\regL)$ with $\varphi(p)=1$ and put
\begin{equation}\label{eq:137}  
 f(q):= \varphi(q) \bigl(\gL(q) \phi_++ \phi_-).
\end{equation}
Then $f\in C_0(X,\dom(\Dmax))$ with $\ga_+\circ f= \gL\cdot (\ga_-\circ f)$,
hence $f\in\dom(\TLa)$ and thus $f(p)\in \dom(T_{\gL,0}^p)$.
$\ga_-(f(p))=1$ proving $\dom(T_{\gL,0}^p)=\dom(D_{\gL(p)})$.

Next consider $\TLa^*\subset \Tmin^*=\Tmax$. The inclusion
$(T_{\gL}^*)^p\subset (T_{\gL}^p)^*$ (Lemma \plref{l:LocalizedOperator})
implies
\begin{equation}\label{eq:138}  
 \dom((T_{\gL}^*)^p)\subset \begin{cases} \dom(\Dmax), & p\in\sisuL,\\
                                      \dom(D_{\gL(p)}), & p\in\regL.
                                     \end{cases}              
\end{equation}
The construction of Eq. \eqref{eq:137} can now be adapted to prove
the claim for the case $p\in\regL$. 

If $p \in (\sisuL)^\circ$ then choose
a continuous compactly supported function $\psi \in C_c((\sisuL)^\circ)$ with
$\psi(p)=1$. Then the functions $q\mapsto \psi(q) \cd \phi_\pm$ lie in
the domain of $\dom(\TLa^*)$ since $(\ga_+ \circ f)(q) = 0 = (\ga_- \circ f)(q) =
0$ for all $q \in (\sisuL)^\circ $ and all $f \in \dom(\TLa)$. This proves that the
vectors $\phi_+$ and $\phi_-$ are contained in $\dom\big( (\TLa^*)^p\big)$.

If $p\in\pl(\regL)$ and $f\in\dom\bigl(T_{\gL}^*\bigr)$ with
$\ga_-(f(p))\not=0$ then $f\in C_0(X,\dom(\Dmax))$, hence there is an
open neighborhood $U$ of $p$ such that $\ga_-(f(q))\not=0$ for $q\in U$.
Furthermore $\ga_+(f(q))=\gL(q)\cdot \ga_-(f(q))$ for $q\in U\cap \regL$.
Thus
\begin{equation}\label{eq:152}  
 \tilde\Lambda(p):=\lim\limits_{q\to p,\; q\in \regL} \Lambda(q)
                  =\lim\limits_{q\to p,\; q\in \regL}
                  \frac{\ga_+(f(q))}{\ga_-(f(q))}\in S^1
\end{equation}
exists and thus $\ga_+(f(p))=\tilde\gL(p)\cdot \ga_-(f(p))\not=0$, too.
Thus after possibly making $U$ smaller we may assume that also $\ga_+(f(q))\not=0$
for $q$ in $U$ and hence the continuous function
\begin{equation}\label{eq:139}  
 \tilde\gL(q):=\frac{\ga_+(f(q))\;  |\ga_-(f(q))|}{\ga_-(f(q)) \;|\ga_+(f(q))|},\quad q\in U  
\end{equation}
is a continuous extension of $\gL\restriction U\cap\regL$, proving
that necessarily $p\in\regiL$. This argument proves that $(T_{\gL}^*)^p=\Dmin$
for $p\in \sisurL$. Continuing with $p\in\regiL$ we
first note that taking limits $q\to p, q\in U\cap \regL$ we see that
necessarily $\ga_+(f(p))=\tilde \gL(p)\cdot \ga_-(f(p))$ proving 
$\dom((T_\gL^*)^p)\subset \dom(D_{\tilde \gL(p)})$. To prove the converse
inclusion put (cf. \eqref{eq:137})
\begin{equation}\label{eq:140}  
 f(q):= \varphi(q) \bigl(\tilde \gL(q) \phi_++ \phi_-).
\end{equation}
Then it easily follows that $f\in\dom\bigl(\TLa^*\bigr)$.
Since $\ga_-(f(p))=1$ we conclude that indeed $(\TLa^*)^p =
D_{\tilde\gL(p)}$ and the Lemma is proved.
\end{proof}

\begin{prop}\label{p:RegT}
\begin{thmenum}
\item $\TLa$ is regular if and only if $(\sisuL)^\circ=\sisuL$. 
\item $\TLa$ is selfadjoint if and only if $\ssupp\gL=\ssupp_r\gL$.
\item $\TLa$ is selfadjoint and regular if and only if $\gL$ is continuous (i.e. $\ssupp\gL=\emptyset$).
\item $\TLa^*$ is selfadjoint and regular if and only if $\tilde \gL$ is continuous
(i.e. $\ssupp\gL=\reg_\infty\gL$).
\end{thmenum}
\end{prop}

Hence if $\ssupp\gL=\ssupp_r\gL\not=\emptyset$ then $\TLa$
is selfadjoint and not regular (cf. \cite[Chap. 9]{Lan:HCM}).

If $\ssupp\gL=\reg_\infty\gL\not = \emptyset$ then $\TLa^*$ is
regular and selfadjoint, but $\TLa\subsetneqq \TLa^*=\TLa^{**}$
is not regular (cf. \cite{Pal:ROH}).

\begin{proof}
 1.  By Theorem \plref{t:Pure} $\TLa$ is regular if and only if 
 $(\TLa^p)^*=(\TLa^*)^p$ for all $p\in X$. 
 In view of \eqref{eq:126}, \eqref{eq:126a} this is the case if and only if 
 $(\sisuL)^\circ=\sisuL$.

 2. For $\TLa$ being selfadjoint it is necessary that $\TLa^p=(\TLa^*)^p$ for
 all $p\in X$. By Lemma \plref{l:TLambdap} the latter is only true
 if $\sisuL=\sisurL$. If that is the case let us consider $f\in\dom(\TLa^*)$.
 Then by \eqref{eq:126a} 
 $f\in C_0(X,\dom(\Dmax))$, $f(p)\in\dom(\Dmin)$ if $p\in\sisurL$, and
 $f(p)\in \dom(D_{\gL(p)})$ if $p\in\regL$. But then, since
 $(\sisuL)^\circ=\emptyset$
 we have $\ga_+\circ f= \gL\cdot\ga_-\circ f$, thus $f\in\dom(\TLa)$.

 3. This is just the obvious combination of 1. and 2.
 
 4. This follows as 1. and 2. by applying Lemma \plref{l:TLambdap} to $\TLa^*$.
 \end{proof}

\begin{prop}\label{p:NonRegularFaithfulState}
Assume that $\ovl{\regL}=X$ and let $\mu$ be a probability measure on $X$ with
\begin{thmenum}
\item $\supp \mu = \overline{\regL}=X,$
\item $\pl(\regL)$ is a $\mu$--null set.
\end{thmenum}
Then the representation $\pi_\mu$ is faithful, $\TLa^\mu =(\TLa^\mu)^*$, and
\begin{align}  
  \dom(\TLa^\mu)&=\bigsetdef{f\in \Tmax^\mu}{ f(x) \in\dom(D_{\gL(x)})
  \text{ for $\mu$-a.e. } x\in X}\label{eq:141a}\\
                &=\bigsetdef{f\in \Tmax^\mu}{ \inn{f, \eta_{\gL}^\perp}_{T,\mu}=0},
                \label{eq:141b}
\end{align}
where $\eta_{\gL}^\perp(x):= \eta_{\gL(x)}^\perp$ (cf. \eqref{eq:131}). 

However, if $\pl(\regL)=\ssupp_r\gL\not=\emptyset$ then $\TLa$ is selfadjoint but not
regular while if $\pl(\regL)=\reg_\infty\gL\not=\emptyset$ then $\TLa$ is neither
selfadjoint nor regular but $\TLa^\mu =(\TLa^\mu)^*$ for the \emph{faithful}
representation $\pi_\mu$.
\end{prop}
As an example for the last situation we could concretely
take $X=[0,1]$, $\Lambda:[0,1]\rightarrow S^1$ such that
 $\gL\restriction (0,1] $ is continuous but discontinuous at $0$ (with
not existing limit at $0$ in the first case and existing limit at $0$ in the
second), and $\mu(f):=\int_0^1 f.$

This example shows that the answer to the following question, which the attentive reader
might have hoped to be affirmative, is \emph{negative}:

\begin{problem}
Does the essential selfadjointness of the localized unbounded operator of
the form $T^\pi$ on $H_\pi$ imply the selfadjointness and regularity of the 
closed symmetric operator $T$ when the presentation $\pi$ is 
\textbf{faithful}?
\end{problem}

One might however still be tempted to think that the above statement is true
when the faithful representation is the atomic representation of $\sA$, cf.
Conjectures \plref{c:one}--\plref{c:three} and Theorem \plref{t:LocGlobPure}.

\begin{proof}[Proof of Prop. \plref{p:NonRegularFaithfulState}]
 Since $\supp \mu =X$ it follows that the representation $\pi_\mu$
 of $C(X)$ by multiplication operators on $L^2(X,\mu;H)$ is faithful.

Now, let $\Om := \regL$ and let $\Theta : \Om \to S^1$ denote the restriction of $\Theta$ to $\Om$. We can then consider the operator $T_{\Theta}$ acting on the Hilbert module $C_0(\Om, H)$ together with the localization $T_{\Theta}^{\si}$. Here $\si$ denotes the restriction of the probability measure $\mu$ to $\Om$. We then get from Proposition \ref{p:RegT} that $T_{\Theta}$ is selfadjoint and regular and hence from Theorem \ref{t:LocGlobPure} that the localization $T_{\Theta}^{\si}$ is selfadjoint. The selfadjointness of the localization $T_{\Lambda}^\mu$ now follows by noting that we have the inclusion $T_{\Theta}^{\si} \su T_{\Lambda}^\mu$ under the identification of Hilbert spaces $L^2(\Om,\si;H) \cong L^2(X,\mu;H)$. Here we use that $\mu(\pa \regL) = 0$.

The identities in Eq. \eqref{eq:141a} and Eq. \eqref{eq:141b} are now obvious.
\end{proof}

\section{Sums of selfadjoint regular operators}\label{s:SSRO} 
\marginpar{give a bit more motivation, KK product}
Let us consider a Hilbert $C^*$--module $E$ over some $C^*$--algebra  
$\sA$. Furthermore, let $S$ and $T$ be two selfadjoint and regular operators with domains $\dom (S) \su
E$ and $\dom (T) \su E$ respectively. The main purpose of this section is then
to study the selfadjointness and regularity of the sum operator
\begin{equation}\label{eq:DefD}
D := \begin{pmatrix}
0 & S - i\, T  \\
S + i\, T & 0 
\end{pmatrix}, \quad  \dom (D) = \big( \dom (S) \cap \dom (T) \big)^2\subset
E\oplus E.
\end{equation}
As mentioned in the introduction this question is essential when dealing with
the Kasparov product of unbounded modules. To be more precise we shall see
that the following three assumptions are sufficient for the above sum to be a
selfadjoint and regular operator.

\begin{assu}\label{sumassu}
We will assume that we have a dense submodule $\sE \su E$ such that the
following conditions are satisfied:
\begin{enumerate}
\item The submodule $\sE \su \dom (T)$ is a core for $T$.
\item We have the inclusions
\begin{equation}\label{eq:Assu1}
(S - i \cd \mu)^{-1}(\xi) \in \dom (S) \cap \dom (T)
\quad \emph{and} \quad T(S - i \cd \mu )^{-1}(\xi) \in \dom (S)
\end{equation}
for all $\mu \in \R\sm \{0\}$ and all $\xi \in \sE$.
\item The module homomorphism
\begin{equation}\label{eq:Assu2}
[S,T] \cd (S - i \cd \mu)^{-1} : \sE \to E
\end{equation}
extends to a bounded ($\sA$-linear) operator $X_\mu : E \to E$ between Hilbert
$C^*$--modules for all $\mu \in \R\sm \{0\}$.
\end{enumerate}
\end{assu}

We will start by proving that the sum operator is selfadjoint. Afterwards we
shall apply the Local--Global Principle to prove that the sum operator is
regular as well. The following Lemma will be used several times:
\begin{lemma}\label{l:StrongConv}
Let $P$ be a selfadjoint regular operator in the Hilbert $C^*$--module $E$.
Let $(f_n)_{n\in\N}\subset C_\infty(\R)$ be a sequence of functions (cf.
Section \plref{ss:SRO}) such that 
\begin{thmenum}
\item $\sup_n \| f_n \|_\infty<\infty,$
\item $(f_n)_{n\in \N}$  converges to $f\in C_\infty(\R)$ uniformly on
compact subsets of $\R$. 
\end{thmenum}
Then $f_n(P)$ converges \emph{strongly} to $f(P)$ as $n\to \infty$.
\end{lemma}
\begin{proof} Consider first $x\in\dom(P)$. Then with $\varphi(t):=(t+i)\ii$
 and $y:= (P+ i ) x$ we have $x=\varphi(P)y$. Since $\varphi$ vanishes
 at $\infty$, $f_n \varphi \to f\varphi$ \emph{uniformly}, in particular
 \[
 f_n(P)x = \bigl((f_n\varphi)(P)\bigr) y \longrightarrow f(P) x.
 \]
 The sequence $(f_n(P))_{n\in\N}$ therefore converges
 strongly to $f(P)$ on the dense submodule $\dom(P)$. Since the sequence 
 is uniformly norm bounded this implies the strong convergence.
\end{proof} 
\subsection{Selfadjointness}\label{ss:selfadj}  
Let us assume that the conditions in Assumption \ref{sumassu} are satisfied.

We begin by noting that the commutator $[S,T]$ is densely defined. Indeed, the
domain of $[S,T]$ contains the dense submodule $(S-i)^{-1}(\sE)$. As a
consequence we get that the bounded operators $X_\mu, \,\mu \in \R\sm \{0\},$
defined in Eq. \eqref{eq:Assu2} are adjointable. The adjoint of $X_\mu$ is given by the expression
\begin{equation}\label{eq:XmuAdjoint}
(X_{\mu})^*\xi = - (S + i\cd \mu)^{-1}[S,T]\xi
\end{equation}
for all $\xi$ in the dense submodule $\dom\big( [S,T]\big) \su E$.

We continue by showing that the core $\sE$ can be replaced by the domain of
the selfadjoint and regular operator $T$.
\begin{prop}\label{coret}
The conditions in Assumption \plref{sumassu} are satisfied for $\sE =
\dom (T)$.
\end{prop}
\begin{proof}
Let us fix some vector $\xi \in \dom (T)$ and some number $\mu \in \R\sm
\{0\}$. It then suffices to prove the inclusions
\begin{equation}\label{eq:dominc}
\arr{ccc}{
(S - i \cd \mu)^{-1}(\xi) \in \dom(T) 
& \T{and} & T(S-i \cd \mu)^{-1}(\xi) \in \dom(S).
}
\end{equation}
To this end we let $(\xi_n)\subset \sE$ be some sequence such that
\[
\xi_n \to \xi, \quad \T{and} \quad T\xi_n \to T\xi 
\]
in the norm topology of $E$. It then follows that the sequences
\[
\begin{split}
T(S - i \cd \mu)^{-1}\xi_n & =  
(S - i \cd \mu)^{-1} T\xi_n + (S - i \cd \mu)^{-1} X_\mu\xi_n 
 \T{ and} \\
ST(S - i \cd \mu)^{-1}\xi_n & = 
S (S-i\cd \mu)^{-1} T\xi_n + 
S (S-i \cd \mu)^{-1} X_\mu \xi_n
\end{split}
\]
converge in $E$. But this proves the validity of the inclusions in
\eqref{eq:dominc} since our unbounded operators are closed.
\end{proof}

For later use we state and prove the following:

\begin{lemma}\label{strzero}
The sequence of adjointable operators
\[
R_n := 
\frac{i}{n} \cd \big( \frac{i}{n} S + 1 \big)^{-1} \cd [S,T] 
\cd \big( \frac{i}{n} S + 1\big)^{-1} \in \B(E)
\]
as well as $(R_n^*)_n$ converge strongly to the $0$-operator.
\end{lemma}
\begin{proof}
For each $n \in \N$ we can rewrite $R_n$ as
\[
R_n = \big( \frac{i}{n} S + 1 \big)^{-1} \cd X_{-1} \cd (S + i) (S - i\cd n)^{-1}.
\]
Since the factors in this decomposition are adjointable it follows that $R_n$
is adjointable as well. By Lemma \plref{l:StrongConv} the sequence of adjointable operators
\[
\bigl((S + i) (S - i\cd n)^{-1}\bigr)_{n\in\N} \subset \B(E)
\]
converges strongly to zero as $n\to
\infty$. Furthermore the sequence of adjointable operators
\[
 \Bigl(\big( \frac{i}{n} S + 1 \big)^{-1} \cd X_{-1}\Bigr)_{n\in\N} \subset \B(E)
\]
is uniformly bounded in the operator norm. But these two observations prove
that $R_n\to 0$ strongly. 

The same line of argument applies to the adjoint $R_n^*$ and the Lemma is
proved.
\end{proof}

The $C^*$-algebra estimates of the next two lemmas will be important for
proving that the sum operator $D$ is closed.

\begin{lemma}\label{commest}
There exists a constant $C > 0$ such that we have the following inequality
between selfadjoint elements of the $C^*$--algebra $\sA$
\[
\pm i \cd \binn{[S,T]\xi,\xi} \leq \frac{1}{2}\inn{S\xi,S \xi} + C \inn{\xi,\xi}
\]
for all $\xi \in \dom\big( [S,T] \big)$.
\end{lemma}
\begin{proof}
 The selfadjointness of $S,T$ and Assumption \plref{sumassu} imply that
 $\binn{[S,T]\xi,\xi}$ is skewadjoint for $\xi\in\dom([S,T])$. Furthermore, the inequalities
\[
\begin{split}
\pm 2 i \cd \binn{[S,T]\xi,\xi}
& = \mp \binn{i [S,T] \mu \xi, \mu^{-1} \xi} \mp \binn{\mu^{-1}\xi,i [S,T]\mu \xi}
\\
& \leq \mu^2\, \binn{[S,T]\xi,[S,T] \xi} + \mu^{-2}\inn{\xi,\xi} \\
& \leq \mu^2\, \big\| [S,T](S+i)^{-1}\big\|^2 \cd 
\binn{(S+i)\xi,(S+i)\xi} + \mu^{-2} \inn{\xi,\xi} \\
& \leq \mu^2 \|X_{-1}\|^2 \inn{S\xi,S\xi} + 
\big( \mu^2 \|X_{-1}\|^2 + \mu^{-2} \big) \cd \inn{\xi,\xi}
\end{split}
\]
are valid in the $C^*$-algebra $\sA$ for any $\mu \in \R\sm \{0\}$. 
Here we have used again the inequality \eqref{eq:149}.
Letting $\mu = \frac{1}{\|X_{-1}\|}$ now proves the claim.
\end{proof}

\begin{lemma}\label{DSTest}
There exists a constant $C > 0$ such that
\[
\binn{(S \pm i\,T)\xi,(S \pm i\, T)\xi} 
\geq \frac{1}{2}\inn{S\xi,S\xi} + \inn{T\xi,T\xi} - C \inn{\xi,\xi}
\]
for all $\xi \in \dom(S) \cap \dom(T)$.
\end{lemma}
\begin{proof}
By an application of Lemma \ref{commest} we get that
\begin{equation}\label{eq:estdense}
\begin{split}
\binn{(S \pm i\, T)\xi,(S \pm i\, T)\xi} 
& = \inn{S\xi,S\xi} + \inn{T\xi,T\xi}  \mp i \inn{[S,T]\xi,\xi} \\
& \geq \frac{1}{2}\inn{S\xi,S\xi} + \inn{T\xi,T\xi} - C \inn{\xi,\xi}
\end{split}
\end{equation}
for all $\xi \in \dom\big( [S,T]\big)$. This proves the desired inequality on
the submodule $\dom \big( [S,T] \big) \su \dom(S) \cap \dom(T)$.

Now, let $\xi \in \dom(S) \cap \dom(T)$. We will then look at the elements
\[
\xi_n := \big( \frac{i}{n} S  + 1 \big)^{-1}\xi \in \dom\big( [S,T] \big)
\]
in the domain of $[S,T]$. It follows from Lemma \plref{l:StrongConv}
that $\xi_n \to \xi$ and $S\xi_n \to S\xi$  in $E$ as $n\to\infty$.
On the other hand, we have the identities
\[
\begin{split}
T\xi_n 
& = \big( \frac{i}{n} S  + 1 \big)^{-1} T\xi
+  \frac{i}{n} \cd \big( \frac{i}{n} S  + 1 \big)^{-1} [S,T] \big( \frac{i}{n}
S + 1 \big)^{-1}\xi \\
& = \big( \frac{i}{n} S  + 1 \big)^{-1} T\xi + R_n\xi.
\end{split}
\]
It therefore follows from Lemma \ref{strzero} that 
$T\xi_n \to T\xi$ as $n\to\infty$ in the norm of $E$. 
The Lemma now follows by applying the inequality \eqref{eq:estdense}
to $\xi_n$ and taking limits.
\end{proof}

We are now ready to prove the main result of this subsection.

\begin{prop}\label{p:STAdjoint}
Assume that the conditions in Assumption \plref{sumassu} are satisfied. 
The operators $S\pm i\, T$ with domain $\dom(S\pm i\, T) = \dom(S) \cap \dom(T)$
are closed operators and adjoints of each other, i.e. $\bigl( S \pm i\, T
\bigr)^*=\bigl( S \mp i\, T\bigr)$. In other words, the sum operator
\[
D = \begin{pmatrix}
0 & S - i\, T \\
S + i\, T & 0
\end{pmatrix} 
\] 
with domain
\[
\dom(D) = \bigl( \dom(S) \cap \dom(T) \bigr)^2
\]
is selfadjoint.
\end{prop}
\begin{proof} 
The inequality stated in Lemma \ref{DSTest} shows that
the convergence of a sequence in the graph norm of $S \pm i\, T$ implies 
the convergence of the sequence in the graph norm of $S$ \emph{and}
in the graph norm of $T$ individually. Since the unbounded
operators $S$ and $T$ are closed this proves that $S\pm i\, T$ are closed, too.

To show that $S\pm i\, T$ are adjoints of each other we first note
that for each pair of elements $\xi,\eta \in \dom(S) \cap \dom(T)$ 
we have the identity
\[
\binn{(S + i\, T) \xi ,\eta} = \binn{\xi,(S- i\, T) \eta }.
\]
We therefore only need to prove the inclusion of domains
\begin{equation}\label{eq:150}  
\dom\big((S + i\, T)^* \big) \su \dom(S) \cap \dom(T).
\end{equation}
To this end we let $\xi \in \dom\big((S + i\, T)^* \big)$ be a vector in the
domain of the adjoint. We then define the sequence $(\xi_n)$ by
\[
\xi_n := \big( -\frac{i}{n}S + 1\big)^{-1} \xi  \in \dom(S)
\]
which converges to $\xi$ in the norm of $E$. We shall prove that $\xi_n \in
\dom(S) \cap \dom(T)$ for all $n \in \nn$ and that 
$(S+i\, T)^*\xi_n = (S - i\, T)\xi_n$ converges as $n\to\infty$ in the norm of $E$. This will
prove the inclusion \eqref{eq:150} since $S - i\, T$ is already proved to be closed on $\dom(S) \cap \dom(T)$.

We start by proving that $\xi_n \in \dom(S) \cap \dom(T)$. 
Let $\eta \in \dom(T)$. We can then calculate as follows, using Assumption
\plref{sumassu} and Proposition \plref{coret}
\begin{equation}\label{eq:tadj}
\begin{split}
 \binn{\xi_n&, T\eta} =\binn{\big(-\frac{i}{n}S + 1\big)^{-1} \xi , T\eta } \\
   & = \binn{\xi, \big( \frac{i}{n}S + 1\big)^{-1}T \eta } \\
   & = \binn{\xi, T \big( \frac{i}{n}S + 1\big)^{-1}\eta}
      -\binn{\xi, \frac{i}{n}\big( \frac{i}{n}S + 1\big)^{-1}[S,T] 
           \big( \frac{i}{n}S + 1\big)^{-1} \eta } \\
   & = -i \binn{\xi , (S + i\, T) \big( \frac{i}{n}S + 1\big)^{-1} \eta }
       +i \binn{\xi , S \big( \frac{i}{n}S + 1\big)^{-1} \eta }
             - \binn{R_n^* \xi , \eta} \\
   & = -i \binn{\big( -\frac{i}{n}S + 1\big)^{-1}(S+i\, T)^* \xi ,\eta}
       +i \binn{S \xi_n,\eta} - \binn{R_n^*\xi, \eta}.
\end{split}
\end{equation}
Since $T$ is selfadjoint this proves that
\[
\xi_n= \big( -\frac{i}{n}S + 1\big)^{-1} \xi \in \dom(T)
\]
and 
\begin{equation}\label{eq:151}  
   T\xi_n =   i \big( -\frac{i}{n}S + 1\big)^{-1}(S+i\, T)^* \xi -i\, S\xi_n  - R_n^*\xi.
\end{equation}
We end the proof of the inclusion \eqref{eq:150} by showing that
$(S + i\, T)^* \xi_n = (S - i\, T)\xi_n$
converges in the norm of $E$.
From \eqref{eq:151} we infer, since $\xi_n\in\dom(S)\cap\dom(T)$,
\[
(S+i\, T)^*\xi_n = (S-i\, T)\xi_n= \big( -\frac{i}{n}S + 1\big)^{-1}(S+i\, T)^* \xi
                        + i (R_n)^* \xi .
\]
The convergence of the sequence $\{(S+ i\, T)^* \xi_n \}$ now follows from Lemma
\ref{strzero} and $\xi\in\dom(S)\cap\dom(T)$ is proved.

The last claim about $D$ is now clear.
\end{proof}
\subsection{Regularity}  
We still assume that the conditions in Assumption \ref{sumassu} are
satisfied. It is now our goal to improve Proposition \ref{p:STAdjoint} by showing
that the operators $S\pm i\, T$ and hence the sum operator $D$ are also regular. This
will turn out to be an easy consequence of the Local--Global Principle and the
results in Section \ref{ss:selfadj}. We start by noting that the localizations
of $T$ and $S$ satisfy the conditions in Assumption \ref{sumassu}. Remark that
these localizations are selfadjoint by Theorem \ref{t:locglob}.

\begin{lemma}\label{localsas}
Let $\om \in S(\sA)$ be a state on the $C^*$--algebra $\sA$. Then the
localizations with respect to $\om$ of the selfadjoint regular operators $T$
and $S$ satisfy the conditions in Assumption \plref{sumassu}. The desired core
$\sE^\om \su \dom(T^\om)$ is the subspace $\sE^\om := \iom\big(\dom(T)\big)$.
\end{lemma}
\begin{proof}
This is a straightforward consequence of the definitions and Proposition
\ref{coret}.
\end{proof}

In order to prove the regularity of the sum operator $D$ we study the
behavior of localization with respect to the sum operation.

\begin{lemma}\label{locandsum}
Let $\om \in S(\sA)$ be a state on the $C^*$--algebra $\sA$. We then have the
identity
\[
S^\om + i\, T^\om = (S + i\, T)^\om
\]
between localized operators.
\end{lemma}
\begin{proof} Recall that by definition $\dom(S^\go+ i
 T^\go)=\dom(S^\go)\cap\dom(T^\go)$. We start by noting that
\[
(S + i\, T)^\om \su S^\om + i\, T^\om.
\]
This inclusion is valid since $(S + i\, T)^\om_0 \su S^\om_0 + i\, T^\om_0$ and
since $S ^\om + i\, T^\om$ is closed by Proposition \ref{p:STAdjoint}. Thus we only
need to prove the inclusion
\begin{equation}\label{eq:sumincloc}
\dom\big( S^\om + i\, T^\om\big) \su \dom\big( (S + i\, T)^\om \big)
\end{equation}
of domains. 
In order to establish this we first prove that
\begin{equation}\label{eq:incloc}
(S^\om + i \mu)^{-1} \big( \dom(T^\om) \big) \su \dom\big( (S + i\, T)^\om\big)
\end{equation}
for all $\mu \in \R\sm \{0\}$. Let $\mu \in \R\sm \{0\}$ and let $\xi \in
\dom(T^\om)$. We can then find a sequence $(\eta_n)\subset \dom(T)$ such
that 
\[
\iom(\eta_n) \to \xi \quad \T{ and }\quad \iom\big(T \eta_n \big) \to T^\om(\xi).
\]
By Assumption \ref{sumassu} we get that
\[
(S^\om + i \mu)^{-1}\big( \iom(\eta_n) \big)
= \iom\big( (S + i \mu)^{-1} \eta_n  \big) \in \iom\big( \dom(S) \cap \dom(T)
\big) \su \dom\big(  (S + i\, T)^\om\big)
\]
and by continuity we have that
\[
(S^\om + i \mu)^{-1}\big( \iom(\eta_n) \big) \to (S^\om + i \mu)^{-1}\xi.
\]
We therefore only need to prove that the sequence
\[
(S + i\, T)^\om (S^\om + i \mu)^{-1}\big( \iom(\eta_n) \big)
= \iom\big( (S + i\, T)(S + i \mu)^{-1} \eta_n  \big)
\]
is convergent in the norm of $E^\om$. But this follows by the argument given
in the proof of Proposition \ref{coret}.

To finish the proof of the inclusion in \eqref{eq:sumincloc} we let
\[
\xi \in \dom(S^\om + i\, T^\om) = \dom(S^\om)\cap \dom(T^\om).
\]
We then define the sequence $(\xi_n)$ by 
\[
\xi_n := \big(\frac{i}{n}S^\om + 1 \big)^{-1} \xi.
\]
By the argument given in the proof of Lemma \ref{DSTest} we have that
\[
\xi_n \to \xi \quad \T{ and }\quad  (S^\om + i\, T^\om)\xi_n  \to (S^\om + i\, T^\om)\xi
\]
where the convergence takes place in the Hilbert space $E^\om$. But this
proves that $\xi \in \dom\big( (S + i\, T)^\om\big)$ since $\xi_n \in \dom\big( (S
+ i\, T)^\om\big)$ for all $n \in \nn$ by the inclusion in \eqref{eq:incloc}.
\end{proof}

We are now ready to prove the main result of this section.

\begin{theorem}\label{t:SAReg}
Assume that the conditions in Assumption \plref{sumassu} are satisfied. Then the
sum operator
\[
D = \begin{pmatrix}
0 & S - i\, T \\
S + i\, T & 0
\end{pmatrix}
\]
with domain $\big( \dom(S) \cap \dom(T) \big)^2$ is selfadjoint and regular.
\end{theorem}
\begin{proof}
Let $\om \in S(\sA)$ be a state on the $C^*$--algebra $\sA$. By the
Local--Global Principle proved in Theorem \ref{t:locglob} we only need to prove
that the localization $D^\om$ agrees with the selfadjoint operator
\[
(D^\om)' := \begin{pmatrix}
0 & S^\om - i\, T^\om \\
S^\om + i\, T^\om & 0
\end{pmatrix}, \quad \dom\big((D^\om)'\big) := \big( \dom(S^\om) \cap \dom(T^\om) \big)^2.
\]
Remark that the selfadjointness of $(D^\om)'$ is a consequence of Lemma
\ref{localsas} and Proposition \ref{p:STAdjoint}. However, by an application of
Lemma \ref{locandsum} we get the identities
\[
D^\om 
= \begin{pmatrix}
                    0 & (S - i\, T)^\om \\
      (S + i\, T)^\om & 0
  \end{pmatrix}
= \begin{pmatrix}
                    0 & S^\om - i\, T^\om \\
       S^\om + i\, T^\om & 0
  \end{pmatrix}
= (D^\om)'
\]
and the theorem is proved.
\end{proof}

\def\cprime{$'$}
\providecommand{\bysame}{\leavevmode\hbox to3em{\hrulefill}\thinspace}
\providecommand{\MR}{\relax\ifhmode\unskip\space\fi MR }
\providecommand{\MRhref}[2]{%
  \href{http://www.ams.org/mathscinet-getitem?mr=#1}{#2}
}
\providecommand{\href}[2]{#2}

\bibliographystyle{amsalpha-lmp.bst}

\end{document}


\title[Erratum to ``A local global principle'']{Erratum to ``A local global principle for regular operators in Hilbert $C^*$-modules''}{%


\author{Jens Kaad and Matthias Lesch}
\address{Department of Mathematics and Computer Science, The University of Southern Denmark, Campusvej 55, DK-5230 Odense M, Denmark}
\email{jenskaad@hotmail.com, kaad@imada.sdu.dk}

\address{Mathematisches Institut,
Universit\"at Bonn,
Endenicher Allee 60,
53115 Bonn,
Germany}
\email{ml@matthiaslesch.de, lesch@math.uni-bonn.de}
\urladdr{www.matthiaslesch.de, www.math.uni-bonn.de/people/lesch}

\date{\today}

\thanks{The first author is partially supported by the Villum Foundation (grant 7423). The second author is partially supported by the Hausdorff Center for Mathematics}



%

\begin{abstract}
We clarify that one of the main results from ``A local global principle for regular operators in Hilbert $C^*$-modules'' is already contained in an earlier paper by Fran\c{c}ois Pierrot. Moreover, in the paper by Pierrot this result is proved in a stronger version that was only stated as a conjecture in our local global principle paper. We also present a short proof, based on Pierrot's argument, of the core issue in both of the two papers namely a result on the separation of a point from a closed submodule. 
\end{abstract}

\maketitle



\section{Introduction}
In our paper ``A local global principle for regular operators in Hilbert $C^*$-modules'' from 2012, \cite{KaLe:LGR}, we stated the conjecture:

\begin{conjecture}\label{c:locglo}
For a closed symmetric unbounded operator $T$ on a Hilbert $C^*$-module $E$, the following statements are equivalent:
\begin{enumerate}
\item $T$ is selfadjoint and regular.
\item For every irreducible representation $(\pi,H_\pi)$ of $\sA$ the localization $T^\pi$ is selfadjoint.
\item For every pure state $\omega$ on $\sA$ the localization $T^\omega$ is selfadjoint.
\end{enumerate}
\end{conjecture}

Moreover, it was noted that this conjecture would be a consequence of the conjecture:

\begin{conjecture}\label{c:separ}
Let $E$ be a Hilbert $C^*$-module over a $C^*$-algebra $\sA$ and let $L\subset E$ be a proper closed submodule. Then for each $x_0\in E\setminus L$ there exists an irreducible cyclic representation $(H_\pi,\xi_\pi)$ of $\sA$ such that $x_0\ot \xi_\pi\not\in L\hot_\sA H_\pi$.
\end{conjecture}

The first of these two conjectures was a stronger version of one of our main theorems:

\begin{theorem}
For a closed symmetric unbounded operator $T$ on a Hilbert $C^*$-module $E$, the following statements are equivalent:
\begin{enumerate}
\item $T$ is selfadjoint and regular.
\item For every representation $(\pi,H_\pi)$ of $\sA$ the localization $T^\pi$ is selfadjoint.
\item For every state $\omega$ on $\sA$ the localization $T^\omega$ is selfadjoint.
\end{enumerate}
\end{theorem}

Later on (after our paper was published) it was brought to our attention by Bram Mesland that Conjecture \ref{c:locglo} had been proved
earlier by Pierrot \cite[Th\'eor\`eme 1.18]{Pie:ORC} and that Conjecture \ref{c:separ} follows from the proof of \cite[Proposition 1.16]{Pie:ORC}. 

The purpose of this erratum is to give proper credit to Pierrot's work which clearly precedes our paper and that therefore he has priority for the \emph{strong} version of the Local Global Principle. Secondly, since we think that Conjecture \ref{c:separ} might be of independent interest we would like to present here a proof which is based on the arguments of the proof of \cite[Proposition 1.16]{Pie:ORC}.
%

Throughout this erratum we apply the notation of the local global principle paper, \cite{KaLe:LGR}.
%
%

\section{Proof of Conjecture \ref{c:separ}}

\begin{theorem}\label{t:separ} Let $E$ be a Hilbert $C^*$-module over a $C^*$-algebra $\sA$ and let $L\subset E$ be a proper closed submodule. Then for each $x_0\in E\setminus L$ there exists an irreducible cyclic representation $(H_\pi,\xi_\pi)$ of $\sA$ such that $x_0\ot \xi_\pi\not\in L\hot_\sA H_\pi$.
\end{theorem}
Note that in the language of states the map $x\mapsto x\ot_\sA\xi_\pi$ agrees with $\iota_\om : E \to E^\om$ under the identification $E^\om \cong E \hot_{\sA} H_\pi$, where $\om$ is the state $a \mapsto \inn{\xi_\pi,\pi(a)\xi_\pi}$.

We first recall some notation and some facts (cf. \cite[Section 2.3]{RaWi:MEC}): 
$E$ is a $\sK(E)$--$\sA$ bimodule. 
\[
_{\sK(E)}\inn{x,y} := \Theta_{x,y}:= x \inn{y,\cdot}
\]
is a $\sK(E)$--valued inner product (antilinear in the second argument) such that for $x,y,z\in E$ 
\[
_{\sK(E)}\inn{x,y} z:= x\inn{y,z}_{\sA}.
\]
As a $\sK(E)$--module $E$ is automatically \emph{full}. $E^*:=\sK(E,\sA)$ denotes the dual $\sA$--$\sK(E)$ module. Via the
antilinear identification $x\leftrightarrow x^*:=\inn{x,\cdot}_\sA$ this is $E$ where $\sA$ acts on the left by $a\cdot x^*=(xa^*)^*$ and
$\sK(E)$ acts on the right by $x^*\cdot K:= (K^*x)^*$. Furthermore, the inner products are $_{\sA}\inn{x^*,y^*}:=\inn{x,y}_\sA$ and $\inn{x^*,y^*}_{\sK(E)} := \Theta_{x,y} = _{\sK(E)}\inn{x,y}$.


Let $(H_\rho,\xi_\rho)$ be a cyclic irreducible representation of
$\sK(E)$. Then the natural left action of $\sA$ on $E^*$ gives a representation of $\sA$ on the Hilbert space
$H_\pi:=E^*\hot_{\sK(E)} H_\rho$, \cite{RaWi:MEC}, call this the
\emph{induced} representation (induced by $\rho$ and $E$). 
Repeating the construction we obtain the \emph{induced--induced}
representation of $\sK(E)$ on $E\hot_\sA E^*\hot_{\sK(E)} H_\rho$.
It is important to note that the map
\begin{equation}\label{EqIntertwiner}
  \Psi:
    E\hot_\sA E^*\hot_{\sK(E)} H_\rho \longrightarrow H_\rho,
    \quad x\ot y^* \ot \xi_\rho \mapsto \rho( \Theta_{x,y}) \xi_\rho = 
    \rho(_{\sK(E)}\inn{x,y} )\xi_\rho       
\end{equation}
is a unitary isomorphism which intertwines the induced--induced
representation and $\rho$. This isomorphism depends on the 
$\sK(E)$--fullness but not on the $\sA$--fullness.

The irreducibility of $\rho$ therefore implies the irreducibility of the induced representation $\pi$. Namely, if $\pi$ were reducible then the induced--induced representation of $\sK(E)$ on $E\hot_\sA H_\pi$ would be reducible as well and the intertwining isomorphism \eqref{EqIntertwiner} would then yield the reducibility of $\rho$ which is a contradiction.

\begin{proof}[Proof of Theorem \ref{t:separ}]
By \cite[Proposition 2.31]{RaWi:MEC} there is a unique $y_0\in E$ such
that $x_0 = y_0 \inn{y_0,y_0}_\sA = \Theta_{y_0,y_0}(y_0)=:\ov{\Theta}(y_0)$. 
Consider the right ideal $\sK(E,L)=L\hot_\sA E^*\subset \sK(E)$.
The operator $\ov{\Theta}$ has $x_0\not\in L$ in its range, hence is
not in $\sK(E,L)$. Now choose a pure state $\Phi$ of $\sK(E)$ with
$\Phi(T^*T)=0$ for all $T$ in the \emph{left} ideal
$\sK(E,L)^*=\sK(L,E)$ and such that $\Phi(\ov{\Theta})\not=0$, \cite[Lemma 3.13.5]{Ped:CAG}.
Note that by the Cauchy--Schwarz inequality, $\Phi$ then vanishes on $\sK(E,L)$ and on $\sK(L,E)$.
Since $\ov{\Theta}^3=\Theta_{x_0,x_0}$ we also have $\Phi(\Theta_{x_0,x_0})\not=0$ and therefore $\Theta_{x_0,x_0}$ is
not in $\sK(L,E)\cup\sK(E,L)$.
%

Let $(H_\Phi,\xi_\Phi)$ be the corresponding irreducible cyclic representation
and $(H_\pi:=E^*\hot_{\sK(E)} H_\Phi,\pi)$ the induced irreducible
representation of $\sA$. Since
\[
    \inn{x_0^*\ot\xi_\Phi, x_0^* \ot\xi_\Phi}_{H_\pi}
       = \inn { \xi_\Phi, \Theta_{x_0,x_0}\xi_{\Phi}}_{H_\Phi}
       = \Phi(\Theta_{x_0,x_0})\not =0 
\]
we conclude that the vector $\xi_\pi:= (x_0^*\ot\xi_\Phi) / \sqrt{\Phi(\Theta_{x_0,x_0})}$ is not $0$ in
$H_\pi$ and hence, since $\pi$ is irreducible, it is a cyclic vector
for $\pi$. We note in passing that therefore the state
\[
    \om(a):= \inn{x_0^*\ot \xi_\Phi, \pi(a) (x_0^* \ot \xi_\Phi)}_{H_\pi} / \Phi(\Theta_{x_0,x_0})
    = \Phi(\Theta_{x_0 a,x_0}) / \Phi(\Theta_{x_0,x_0})
\]
is a \emph{pure} state on $\sA$.

We now claim that 
\begin{enumerate}
  \item Under the isomorphism \eqref{EqIntertwiner} the subspace
    $L\hot_\sA H_\pi$ is mapped into the orthogonal complement of
    the cyclic vector $\xi_\Phi$, while
  \item the vector $x_0\ot \xi_\pi$ is not mapped into the orthogonal
    complement of the cyclic vector $\xi_\Phi$.
\end{enumerate}
Needless to say, 1. and 2. imply that $x_0\ot \xi_\pi$ is not in
$L\hot_\sA H_\pi$ and hence the theorem is proved.

To prove 1. and 2. we note that for any two vectors $x,y\in E$ we have
\[
      \inn{\xi_\Phi, \Psi(x\ot y^*\ot \xi_\Phi)}
      = \inn{\xi_\Phi,\Theta_{x,y}\xi_\Phi}
           = \Phi(\Theta_{x,y}).
\]
If $x\in L $ or $y\in L$ then by construction $\Phi(\Theta_{x,y})=0$
while for $x=y=x_0$ we have $\Phi(\Theta_{x_0,x_0})\not=0$. Since
$x_0\ot x_0^*\ot\xi_\Phi= (x_0\ot \xi_\pi) \cd \sqrt{\Phi(\Theta_{x_0,x_0})}$ we are done.
\end{proof}
%
%
Remark that the vector $y_0$ was only needed to see that $\Theta_{x_0,x_0}$ does
not lie in $\sK(L,E)\cup\sK(E,L)$. So a posteriori the situation is
quite simple: choose a pure state $\Phi$ of $\sK(E)$ with
$\Phi(\Theta_{x_0,x_0})\not=0$ which vanishes on
$\sK(L,E)\cup\sK(E,L)$. Then $\om(a):=\Phi(\Theta_{x_0a,x_0}) / \Phi(\Theta_{x_0,x_0})$ is
a pure state of $\sA$ with the additional property that
$\iota_\om(x_0)$ is not in the closure of $\iota_\om(L)$.

%


\providecommand{\bysame}{\leavevmode\hbox to3em{\hrulefill}\thinspace}
\providecommand{\MR}{\relax\ifhmode\unskip\space\fi MR }
\providecommand{\MRhref}[2]{%
  \href{http://www.ams.org/mathscinet-getitem?mr=#1}{#2}
}
\providecommand{\href}[2]{#2}